\documentclass[preprint,3p,times]{elsarticle}

\usepackage{amssymb}
\usepackage{amsmath}
\usepackage{amsthm}
\usepackage{color}
\usepackage{enumitem,verbatim}
\usepackage{tikz}
\usepackage{appendix}
\usepackage[colorlinks=true]{hyperref}

\theoremstyle{plain}
  \newtheorem{thm}{Theorem}[subsection]
  \newtheorem{lem}[thm]{Lemma}
  \newtheorem{prop}[thm]{Proposition}
  \newtheorem{cor}[thm]{Corollary}
\theoremstyle{definition}
  \newtheorem{defn}[thm]{Definition}
  \newtheorem{exmp}[thm]{Example}
  \newtheorem{rem}[thm]{Remark}
  
  \newtheorem{ques}[thm]{Question}

\DeclareMathAlphabet{\mathcal}{OMS}{cmsy}{m}{n}
\DeclareMathOperator{\bd}{bd}

\DeclareMathOperator{\epi}{epi}

\makeatletter
\def\ps@pprintTitle{%
 \let\@oddhead\@empty
 \let\@evenhead\@empty
 \def\@oddfoot{\centerline{\thepage}}%
 \let\@evenfoot\@oddfoot}
\makeatother

\newcommand{\ra}{\rightarrow}

\newcommand{\bv}{\bigvee}
\newcommand{\bw}{\bigwedge}

\newcommand{\dbv}{\displaystyle\bv}

\newcommand{\dprod}{\displaystyle\prod}

\renewcommand{\phi}{\varphi}
\newcommand{\al}{\alpha}
\newcommand{\be}{\beta}

\newcommand{\Up}{\Upsilon}
\newcommand{\ep}{\epsilon}

\newcommand{\Lam}{\Lambda}
\newcommand{\lam}{\lambda}

\newcommand{\CB}{\mathcal{B}}
\newcommand{\CC}{\mathcal{C}}

\newcommand{\CF}{\mathcal{F}}

\newcommand{\sk}{{\sf k}}

\newcommand{\bbR}{\mathbb{R}}

\newcommand{\Fk}{\CF^n_{\sf k}}
\newcommand{\Fs}{\CF^n_{\sf s}}
\newcommand{\Ci}{\CC^n_{\sf i}}
\newcommand{\ku}{\sk(u)}

\newcommand{\cu}{{\sf c}(u)}

\newcommand{\su}{{\sf s}(u)}

\renewcommand{\th}{\tilde{h}}
\newcommand{\dsum}{\displaystyle\sum}

\numberwithin{equation}{section}

\allowdisplaybreaks

\begin{document}

\begin{frontmatter}



\title{Fuzzy vectors via convex bodies}


\author[D]{Cheng-Yong Du}
\ead{cyd9966@hotmail.com}

\author[S]{Lili Shen\corref{cor}}
\ead{shenlili@scu.edu.cn}

\cortext[cor]{Corresponding author.}
\address[D]{School of Mathematical Sciences and V.\thinspace C. {\rm\&} V.\thinspace R. Key Lab, Sichuan Normal University, Chengdu 610068, China}
\address[S]{School of Mathematics, Sichuan University, Chengdu 610064, China}

\begin{abstract}
In the most accessible terms this paper presents a convex-geometric approach to the study of fuzzy vectors. Motivated by several key results from the theory of convex bodies, we establish a representation theorem of fuzzy vectors through support functions, in which a necessary and sufficient condition for a function to be the support function of a fuzzy vector is provided. As applications, symmetric and skew fuzzy vectors are postulated, based on which a Mare{\v s} core of each fuzzy vector is constructed through convex bodies and support functions, and it is shown that every fuzzy vector over the $n$-dimensional Euclidean space has a unique Mare{\v s} core if, and only if, the dimension $n=1$.
\end{abstract}


\end{frontmatter}

\section{Introduction}

Since Zadeh introduced the concept of fuzzy sets \cite{Zadeh1965} in the 1960s, \emph{fuzzy numbers}, as a special kind of fuzzy subsets of the set $\bbR$ of real numbers, have received considerable attention both in the theory and the applications of fuzzy sets \cite{Dubois1978,Dubois1980,Dubois1982,Dubois1982a,Dubois1982b,Dubois1993,Dijkman1983,Diamond1994}. The notion of fuzzy number may be generalized to \emph{fuzzy vector} (also \emph{$n$-dimensional fuzzy number}) without obstruction, simply by replacing $\bbR$ with the $n$-dimensional Euclidean space $\bbR^n$ in its definition, which has been widely studied as well \cite{Kaleva1985,Puri1985,Diamond1989,Zhang2001,Zhang2002b,Wang2002,Wang2007,Maeda2008}.

\emph{Convex geometry}, as an independent branch of mathematics, has a much longer history that dates back to the turn of the 20th century \cite{Bonnesen1934}, in which several contributions can be even traced back to the ancient works of Euclid and Archimedes. As a well-developed theory in the past decades, convex geometry has been applied to different areas of geometry, analysis and computer science \cite{Gruber1993,Koldobsky2008}.

It is well known that fuzzy vectors can be characterized through their \emph{level sets} (see Theorem \ref{vec-cb}, originated from \cite{Negoita1975,Kaleva1987}). In particular, each level set of a fuzzy vector is a nonempty, compact and convex subset of $\bbR^n$, which is precisely a \emph{convex body} \cite{Gruber2007,Schneider2013} in the sense of convex geometry. It is then natural to consider the possibility of exploiting the powerful arsenal of convex geometers in the realm of fuzzy vectors, and it is the motivation of this paper. Being tailored to the readership of fuzzy set theorists, the geometric machinery involved in this paper are presented in the most accessible terms, so that hopefully, even a reader who is not familiar with the extensive apparatus of convex geometry could easily follow up.

Specifically, inspired by several key results from the theory of convex bodies, this paper is intended to represent fuzzy vectors through support functions (Section \ref{Representation}) and, as applications, investigate Mare{\v s} cores of fuzzy vectors (Section \ref{Mares}). The backgrounds and our main results are illustrated as follows.

\subsection{Representation of fuzzy vectors via support functions} 

Support functions play an essential role in the study of fuzzy vectors. Explicitly, the \emph{support function} \cite{Puri1985,Diamond1989} of a fuzzy vector $u$ is given by
$$h_u:[0,1]\times S^{n-1}\to\bbR,\quad h_u(\al,x):=\bv_{t\in u_{\al}}\langle t,x\rangle,$$
where $S^{n-1}$ is the unit sphere in $\bbR^n$, $\langle -,-\rangle$ refers to the standard Euclidean inner product of $\bbR^n$, and $u_{\al}$ is the $\al$-level set of $u$. The following question arises naturally:

\begin{ques} \label{Question}
Can we find a necessary and sufficient condition for a function 
$$h:[0,1]\times S^{n-1}\to\bbR$$
to be the support function of a (unique) fuzzy vector?
\end{ques}

This question is partially answered by Zhang--Wu in \cite{Zhang2002b}, where several sufficient conditions are provided, though neither of them is necessary. In order to fully solve this question, several key results from the theory of convex bodies are exhibited in Subsection \ref{CB-Supp}:
\begin{itemize}
\item There exists a hyperplane that supports a convex body at any of its boundary point (Theorem \ref{convex-body-support-hyperplane}).
\item A function $h:S^{n-1}\to\bbR$ is the support function of a (unique) convex body if, and only if, it is \emph{sublinear} (Theorem \ref{cb-supp}).
\end{itemize}

Without assuming any a-priori background by the reader on convex geometry, in Subsection \ref{CB-Supp} we develop all needed ingredients from scratch for the self-containment of this paper. Then, based on the well-known characterization of fuzzy vectors through convex bodies (Theorem \ref{vec-cb}), a representation theorem of fuzzy vectors via support functions is established (Theorem \ref{vec-supp}). Explicitly, it is shown in Theorem \ref{vec-supp} that a function
$$h:[0,1]\times S^{n-1}\to\bbR$$
is the support function of a (unique) fuzzy vector if, and only if, 
\begin{enumerate}[label={\rm(VS\arabic*)}]
\item $h(\al,-):S^{n-1}\to\bbR$ is sublinear for each $\al\in[0,1]$,
\item $h(-,x):[0,1]\to\bbR$ is non-increasing, left-continuous on $(0,1]$ and right-continuous at $0$ for each $x\in S^{n-1}$.
\end{enumerate}
Therefore, a perfect answer is provided for Question \ref{Question}.

\subsection{Mare{\v s} cores of fuzzy vectors}

In the recent works of Qiu--Lu--Zhang--Lan \cite{Qiu2014} and Chai--Zhang \cite{Chai2016}, a crucial property of fuzzy numbers regarding their \emph{Mare{\v s} cores} is revealed. Explicitly, a fuzzy number $u$ is \emph{skew} \cite{Chai2016} if it cannot be written as the sum of a fuzzy number and a non-trivial \emph{symmetric} fuzzy number in the sense of Mare{\v s} \cite{Marevs1989}; that is, if
$$u=v\oplus w$$
and $w$ is symmetric, then $w$ is constant at $0$. A fuzzy number $v$ is the \emph{Mare{\v s} core} \cite{Marevs1992,Hong2003} of a fuzzy number $u$ if $v$ is skew and $u=v\oplus w$ for some symmetric fuzzy number $w$. The following theorem combines the main results of \cite{Qiu2014} and \cite{Chai2016}:

\begin{thm} \label{fuzzy-number-main}
Every fuzzy number has a unique Mare{\v s} core, so that every fuzzy number can be decomposed in a unique way as the sum of a skew fuzzy number, given by its Mare{\v s} core, and a symmetric fuzzy number.
\end{thm}

It is then natural to ask whether it is possible to establish the $n$-dimensional version of Theorem \ref{fuzzy-number-main} for general fuzzy vectors. Unfortunately, a negative answer will be given in Section \ref{Mares} (Theorem \ref{main}).

Based on the representation of fuzzy vectors through convex bodies and support functions in Section \ref{Representation}, Theorem \ref{vec-sum-supp} describes the sum of fuzzy vectors defined by Zadeh's extension principle through the \emph{Minkowski sum} of convex bodies and the sum of their support functions, which is the cornerstone of the results of Section \ref{Mares}. Then, in Subsection \ref{Vec-Sym}, the notion of \emph{symmetric fuzzy vector} is defined as \emph{symmetric around the origin} in accordance with the case of $n=1$ (cf. \cite[Remark 2.1]{Qiu2014}); using the language of convex bodies and support functions, a fuzzy vector is symmetric whenever its level sets are closed balls centered at the origin, or whenever the support functions of its level sets are constant (Theorem \ref{sym-vec-supp}).

The notion of symmetric fuzzy vector allows us to postulate \emph{skew fuzzy vectors} and \emph{Mare{\v s} cores} of fuzzy vectors in Subsection \ref{Vec-Skew}. However, for the purpose of studying their properties we have to be familiar with the \emph{inner parallel bodies} of convex bodies, and this is the subject of Subsection \ref{IPB}, in which we characterize inner parallel bodies through support functions, and prove that every convex body can be uniquely decomposed as the Minkowski sum of an irreducible convex body and a closed ball centered at the origin (Theorem \ref{cb-decomp}).

The first part of Subsection \ref{Vec-Skew} is devoted to the decomposition
\begin{equation} \label{u=cu-oplus-su}
u=\cu\oplus\su
\end{equation}
of each fuzzy vector $u\in\CF^n$, where $\cu$ is a Mare{\v s} core of $u$, and $\su$ is a symmetric fuzzy vector (Theorem \ref{cu-core}). However, unlike Theorem \ref{fuzzy-number-main} for the case of $n=1$, Example \ref{vec-deomp-not-unique} reveals that Equation \eqref{u=cu-oplus-su} may not be the unique way of decomposing a fuzzy vector. Hence, we indeed obtain a negative answer to the possibility of establishing the $n$-dimensional version of Theorem \ref{fuzzy-number-main}, which is stated as Theorem \ref{main}.

Finally, we investigate \emph{Mare{\v s} equivalent} fuzzy vectors in Subsection \ref{Mares-Equiv}. As we shall see, comparing with Mare{\v s} equivalent fuzzy numbers (see Corollary \ref{fuzzy-number-cu=ku}), the Mare{\v s} equivalence relation of fuzzy vectors may behave in quite different ways. As Example \ref{ku-not-Mares-core} reveals, the smallest fuzzy vector $\ku$ of the Mare{\v s} equivalence class of a fuzzy vector $u$ may not be a Mare{\v s} core of $u$, and different skew fuzzy vectors may be Mare{\v s} equivalent to each other.

\section{Representation of fuzzy vectors via support functions} \label{Representation}

\subsection{Convex bodies via support functions} \label{CB-Supp}

Throughout, let $\bbR^n$ denote the $n$-dimensional Euclidean space. Following the terminologies of convex geometry \cite{Gruber2007,Schneider2013}, by a \emph{convex body} in $\bbR^n$ we mean a nonempty, compact and convex subset of $\bbR^n$; that is, $A\subseteq\bbR^n$ is a convex body if it is nonempty, closed, bounded and
$$\lam s+(1-\lam)t\in A$$
whenever $s,t\in A$ and $\lam\in[0,1]$. The set of all convex bodies in $\bbR^n$ is denoted by $\CC^n$.

Let $\langle -,-\rangle$ denote the standard Euclidean inner product of $\bbR^n$. A \emph{hyperplane} $H$ in $\bbR^n$ is usually denoted by
\begin{equation} \label{hyperplane-def}
H=\{t\in\bbR^n\mid \langle t,x_0\rangle=\al\}
\end{equation}
for some $x_0\in\bbR^n\setminus\{o\}$ and $\al\in\bbR$, where $o$ is the origin of $\bbR^n$, and $x_0$ is called a \emph{normal vector} of $H$. Each hyperplane $H$ given by \eqref{hyperplane-def} divides $\bbR^n$ into two closed \emph{halfspaces}
\begin{equation} \label{halfspace-def}
H^-:=\{t\in\bbR^n\mid \langle t,x_0\rangle\leq\al\}\quad\text{and}\quad H^+:=\{t\in\bbR^n\mid \langle t,x_0\rangle\geq\al\}.
\end{equation}
%We say that $H$ \emph{separates} two subsets of $\bbR^n$, if one of them is contained in $H^-$ and the other is contained in $H^+$.

Let $A\subseteq\bbR^n$ be closed and convex. For every $x\in\bbR^n$, there exists a unique point $p_A(x)\in A$ such that
\begin{equation} \label{pA-def}
||x-p_A(x)||=d(x,A):=\bw_{a\in A}||x-a||,
\end{equation}
where $||\text{-}||$ refers to the standard Euclidean norm on $\bbR^n$. Indeed, the existence of $p_A(x)$ is obvious by the closedness of $A$. For the uniqueness of $p_A(x)$, suppose that $q_A(x)\in A$ also satisfies \eqref{pA-def}, but $p_A(x)\neq q_A(x)$. Then $\dfrac{p_A(x)+q_A(x)}{2}\in A$ by the convexity of $A$, but
$$\left|\left|x-\dfrac{p_A(x)+q_A(x)}{2}\right|\right|=\left|\left|\dfrac{x-p_A(x)}{2}+\dfrac{x-q_A(x)}{2}\right|\right|
<\left|\left|\dfrac{x-p_A(x)}{2}\right|\right|+\left|\left|\dfrac{x-q_A(x)}{2}\right|\right|=||x-p_A(x)||=||x-q_A(x)||;$$
that is, $\dfrac{p_A(x)+q_A(x)}{2}$ is strictly closer to $x$ than $p_A(x)$ and $q_A(x)$, contradicting to the fact that $p_A(x)$ and $q_A(x)$ both satisfy Equation \eqref{pA-def}. Thus we obtain a well-defined map
$$p_A:\bbR^n\to A,$$
called the \emph{metric projection} of $A$. It is obvious that $p_A(a)=a$ for all $a\in A$.

\begin{lem} (See \cite{Gruber2007,Schneider2013}.) \label{metric-projection-non-expansive}
If $A\subseteq\bbR^n$ is closed and convex, then the metric projection $p_A:\bbR^n\to A$ is \emph{non-expansive} in the sense that
$$||p_A(x)-p_A(y)||\leq||x-y||$$
for all $x,y\in\bbR^n$.
\end{lem}

\begin{proof}
We only prove the case of $x,y\in\bbR^n\setminus A$, while the rest cases can be treated analogously. Since the conclusion holds trivially when $p_A(x)=p_A(y)$, suppose that $p_A(x)\neq p_A(y)$. In this case, the convexity of $A$ guarantees that the line segment $[p_A(x),p_A(y)]\subseteq A$. Considering the hyperplane
$$H_x:=\{t\in\bbR^n\mid\langle t,p_A(x)-p_A(y)\rangle=\langle p_A(x), p_A(x)-p_A(y)\rangle\},$$
it is clear that $p_A(y)\in H_x^-$. We claim that $x\in H_x^+$. In fact, if $x\in\bbR^n\setminus H_x^+$, then $\langle x-p_A(x),p_A(y)-p_A(x)\rangle>0$, and consequently the angle between $x-p_A(x)$ and $p_A(y)-p_A(x)$ is acute. This means that the line segment $[p_A(x),p_A(y)]$ contains a point which is strictly closer to $x$ than $p_A(x)$, contradicting to the definition of $p_A(x)$ (see Equation \eqref{pA-def}). Similarly, for the hyperplane
$$H_y:=\{t\in\bbR^n\mid\langle t,p_A(x)-p_A(y)\rangle=\langle p_A(y), p_A(x)-p_A(y)\rangle\}$$
we may deduce that $p_A(x)\in H_y^+$ and $y\in H_y^-$. 

Since $p_A(x)\in H_x$, $p_A(y)\in H_y$ and $p_A(x)-p_A(y)$ is the normal vector of both the hyperplanes $H_x$ and $H_y$, the distance between $H_x$ and $H_y$ is precisely $||p_A(x)-p_A(y)||$. From $x\in H_x^+$, $p_A(y)\in H_x^-$, $p_A(x)\in H_y^+$ and $y\in H_y^-$ we see that the distance between $x$ and $y$ is no less than the distance between $H_x$ and $H_y$; that is, $||p_A(x)-p_A(y)||\leq||x-y||$.
\end{proof}

Let $A\subseteq\bbR^n$ be a subset, and let $H\subseteq\bbR^n$ be a hyperplane. We say that $H$ \emph{supports} $A$ at $t_0$ if $t_0\in A\cap H$ and either $A\subseteq H^+$ or $A\subseteq H^-$; in this case, $t_0$ necessarily lies in the \emph{boundary} $\bd A$ of $A$, and $H$ is called a \emph{support hyperplane} of $A$. If a hyperplane $H$ given by \eqref{hyperplane-def} supports $A$ and $A\subseteq H^-$, then $x_0$ is called an \emph{exterior normal vector} of $H$.

\begin{lem} \label{convex-body-support-hyperplane-not-in}
If $A\subseteq\bbR^n$ is closed and convex, then for each $x\in\bbR^n\setminus A$, there exists a hyperplane $H$ that supports $A$ at $p_A(x)$, with $x-p_A(x)$ being its exterior normal vector.
\end{lem}

\begin{proof}
Let
$$H=\{t\in\bbR^n\mid\langle t,x-p_A(x)\rangle=\langle p_A(x),x-p_A(x)\rangle\}.$$
Then $p_A(x)\in A\cap H$, and we claim that $A\subseteq H^-$. In fact, if there exists $z\in A\cap(\bbR^n\setminus H^-)$, then the line segment $[p_A(x),z]\subseteq A$ by the convexity of $A$. Note that $z\in\bbR^n\setminus H^-$ means that $\langle z-p_A(x),x-p_A(x)\rangle>0$, and consequently the angle between $z-p_A(x)$ and $x-p_A(x)$ is acute. Thus, the line segment $[p_A(x),z]$ must contain a point which is strictly closer to $x$ than $p_A(x)$, contradicting to the definition of $p_A(x)$.
\end{proof}

The following theorem is well known in convex geometry, and we present a proof here for the sake of self-containment:

\begin{thm} (See \cite{Gruber2007,Schneider2013}.) \label{convex-body-support-hyperplane}
If $A\subseteq\bbR^n$ is closed and convex, then for each $t_0\in\bd A$, there exists a (not necessarily unique) hyperplane $H$ that supports $A$ at $t_0$.
\end{thm}

\begin{proof}
Let $\{t_k\}\subseteq\bbR^n\setminus A$ be a sequence that converges to $t_0\in\bd A$, which induces a sequence $\{p_A(t_k)\}\subseteq\bd A$. For each positive integer $k$, by Lemma \ref{convex-body-support-hyperplane-not-in} we may find a hyperplane 
$$H_k=\{t\in\bbR^n\mid\langle t,t_k-p_A(t_k)\rangle=\langle p_A(t_k),t_k-p_A(t_k)\rangle\}$$
that supports $A$ at $p_A(t_k)$, with $t_k-p_A(t_k)$ being its exterior normal vector. Let
$$x_k:=\dfrac{t_k-p_A(t_k)}{||t_k-p_A(t_k)||}.$$
Then $x_k$ belongs to $S^{n-1}$, the unit sphere in $\bbR^n$. By the compactness of $S^{n-1}$, the sequence $\{x_k\}$ has a convergent subsequence, and without loss generality we may suppose that $\{x_k\}$ itself converges to $x_0\in S^{n-1}$. Note that
$$||p_A(t_k)-t_0||=||p_A(t_k)-p_A(t_0)||\leq||t_k-t_0||$$
by Lemma \ref{metric-projection-non-expansive}, which necessarily forces $\lim\limits_{k\ra\infty}p_A(t_k)=t_0$ as we already have $\lim\limits_{k\ra\infty}t_k=t_0$. We claim that the hyperplane
$$H=\{t\in\bbR^n\mid\langle t,x_0\rangle=\langle t_0,x_0\rangle\}$$
supports $A$ at $t_0$, with $x_0$ being its exterior normal vector. To see this, note that $t_0\in A\cap H$ is obvious, and for every positive integer $k$ we have $A\subseteq H_k^-$, which implies that
$$\langle a,t_k-p_A(t_k)\rangle\leq\langle p_A(t_k),t_k-p_A(t_k)\rangle,\quad\text{i.e.,}\quad\langle a,x_k\rangle\leq\langle p_A(t_k),x_k\rangle$$
for all $a\in A$. By letting $k\ra\infty$ in the above inequality we immediately obtain that $\langle a,x_0\rangle\leq\langle t_0,x_0\rangle$ for all $a\in A$, and consequently $A\subseteq H^-$, which completes the proof.
\end{proof}

Recall that the \emph{support function} \cite{Gruber2007,Schneider2013} of a convex body $A\in\CC^n$ is given by
\begin{equation} \label{hA-def}
h_A:S^{n-1}\to\bbR,\quad h_A(x):=\bv_{a\in A}\langle a,x\rangle,
\end{equation}
where $S^{n-1}$ is the unit sphere in $\bbR^n$. Obviously, $h_A$ is bounded on $S^{n-1}$; indeed,
$$|h_A(x)|\leq\bv_{a\in A}||a||$$
for all $x\in S^{n-1}$.

\begin{rem} \label{cb-supp-remark}
The domain of the support function of a convex body $A\in\CC^n$ is defined as $\bbR^n$ in \cite{Gruber2007,Schneider2013}; that is,
\begin{equation} \label{hA-def-Rn}
h_A:\bbR^n\to\bbR,\quad h_A(x):=\bv_{a\in A}\langle a,x\rangle.
\end{equation}
In fact, the function $h_A$ given by \eqref{hA-def-Rn} is completely determined by its values on $S^{n-1}$, since it always holds that
$$h_A(o)=0\quad\text{and}\quad h_A(x)=||x||\cdot h_A\left(\dfrac{x}{||x||}\right)$$
for all $x\in\bbR^n\setminus\{o\}$. Therefore, it does no harm to restrict the domain of $h_A$ to $S^{n-1}$.
\end{rem}

Conversely, to each function $h:S^{n-1}\to\bbR$ we may associate a subset
\begin{equation} \label{Ah-def}
A_h:=\{t\in\bbR^n\mid\forall x\in S^{n-1}:\ \langle t,x\rangle\leq h(x)\}
\end{equation}
of $\bbR^n$. Convex bodies can be fully characterized through support functions as follows, which is a modification of \cite[Theorem 4.3]{Gruber2007} and \cite[Theorem 1.7.1]{Schneider2013}:

\begin{thm} \label{cb-supp}
A function $h:S^{n-1}\to\bbR$ is the support function of a (unique) convex body $A_h$ if, and only if, $h$ is \emph{sublinear} in the sense that
$$h(x)+h(-x)\geq 0$$
and
$$h\left(\dfrac{\lam x+(1-\lam)y}{||\lam x+(1-\lam)y||}\right)\leq\dfrac{\lam h(x)+(1-\lam)h(y)}{||\lam x+(1-\lam)y||}$$
for all $x,y\in S^{n-1}$, $\lam\in[0,1]$ with $\lam x+(1-\lam)y\neq 0$.
\end{thm}

Before proving this theorem, let us recall that a function $f:\bbR^n\to\bbR$ is \emph{convex} if
$$f(\lam x+(1-\lam)y)\leq\lam f(x)+(1-\lam)f(y)$$
for all $x,y\in\bbR^n$ and $\lam\in[0,1]$.

\begin{lem} (See \cite{Gruber2007,Schneider2013}.) \label{f-convex-epif-convex}
A function $f:\bbR^n\to\bbR$ is convex if, and only if, its \emph{epigraph}
$$\epi f:=\{(x,\al)\in\bbR^n\times\bbR\mid f(x)\leq\al\}\subseteq\bbR^{n+1}$$
is convex. In this case, $f$ is necessarily continuous on $\bbR^n$ and, consequently, $\epi f$ is closed.
\end{lem}

\begin{proof}
{\bf Step 1.} $f$ is convex if, and only if, $\epi f$ is convex. If $f$ is convex, for any $(x,\al),(y,\be)\in\epi f$ and $\lam\in[0,1]$ we have
$$f(\lam x+(1-\lam)y)\leq\lam f(x)+(1-\lam)f(y)\leq\lam\al+(1-\lam)\be;$$
that is, $\lam(x,\al)+(1-\lam)(y,\be)=(\lam x+(1-\lam)y,\lam\al+(1-\lam)\be)\in\epi f$. Thus $\epi f$ is convex.

Conversely, if $\epi f$ is convex, for any $x,y\in\bbR^n$ and $\lam\in[0,1]$ we have
$$(\lam x+(1-\lam)y,\lam f(x)+(1-\lam)f(y))=\lam(x,f(x))+(1-\lam)(y,f(y))\in\epi f$$
because $(x,f(x)),(y,f(y))\in\epi f$; that is, $f(\lam x+(1-\lam)y)\leq\lam f(x)+(1-\lam)f(y)$, showing that $f$ is convex.

{\bf Step 2.} If $f:\bbR^n\to\bbR$ is convex, then $f$ is continuous on $\bbR^n$. To this end, for any $x_0\in\bbR^n$ we choose a simplex 
$$S=\left\{\dsum\limits_{i=1}^{n+1}\lam_i x_i\mathrel{\Big|}\dsum\limits_{i=1}^{n+1}\lam_i=1\ \text{and}\ \lam_i\geq 0\ \text{for all}\ i=1,\dots,n+1\right\}$$ 
with vertices $x_1,\dots,x_{n+1}\in\bbR^n$, such that there exists an open ball $B(x_0,\rho)\subseteq S$ ($\rho>0$). Note that $f$ is clearly bounded on $S$ since, by Jensen's inequality (cf. \cite[Remark 1.5.1]{Schneider2013}),
$$f(x)=f\left(\dsum\limits_{i=1}^{n+1}\lam_i x_i\right)\leq\dsum\limits_{i=1}^{n+1}\lam_i f(x_i)\leq c:=\max\{f(x_1),\dots,f(x_{n+1})\}$$
for all $x=\dsum\limits_{i=1}^{n+1}\lam_i x_i\in S$. Now, for any $y=x_0+\lam t\in B(x_0,\rho)$ ($\lam\in[0,1]$, $||t||=\rho$), the convexity of $f$ implies that
$$f(y)=f(x_0+\lam t)=f((1-\lam)x_0+\lam(x_0+t))\leq(1-\lam)f(x_0)+\lam f(x_0+t),$$
and consequently $f(y)-f(x_0)\leq \lam(f(x_0+t)-f(x_0))\leq\lam(c-f(x_0))$ because $x_0+t\in S$. Similarly,
$$f(x_0)=f\left(\dfrac{1}{1+\lam}y+\dfrac{\lam}{1+\lam}(x_0-t)\right)\leq\dfrac{1}{1+\lam}f(y)+\dfrac{\lam}{1+\lam}f(x_0-t),$$
and consequently $f(x_0)-f(y)\leq\lam(f(x_0-t)-f(x_0))\leq\lam(c-f(x_0))$. It follows that
$$|f(y)-f(x_0)|\leq\lam(c-f(x_0))=\dfrac{c-f(x_0)}{\rho}||y-x_0||$$
for all $y\in B(x_0,\rho)$, which immediately implies the continuity of $f$ at $x_0$.

{\bf Step 3.} If $f:\bbR^n\to\bbR$ is continuous, then $\epi f$ is closed. This is easy since, for any sequence $\{(x_k,\al_k)\}\subseteq\epi f$ that converges to $(x,\al)$, $f(x)\leq\al$ becomes an immediate consequence of $f(x_k)\leq\al_k$ for all positive integers $k$ in conjunction with the continuity of $f$, which means precisely that $(x,\al)\in\epi f$. This completes the proof.
\end{proof}

\begin{proof}[Proof of Theorem \ref{cb-supp}]
{\bf Necessity.} Let $A\in\CC^n$ be a convex body. Then it is clear that
$$h_A(x)+h_A(-x)=\bv_{a\in A}\langle a,x\rangle+\bv_{a\in A}\langle a,-x\rangle=\bv_{a\in A}\langle a,x\rangle-\bw_{a\in A}\langle a,x\rangle\geq 0$$
and
$$h_A\left(\dfrac{\lam x+(1-\lam)y}{||\lam x+(1-\lam)y||}\right)=\bv_{a\in A}\left\langle a,\dfrac{\lam x+(1-\lam)y}{||\lam x+(1-\lam)y||}\right\rangle
\leq\dfrac{\lam\dbv_{a\in A}\langle a,x\rangle+(1-\lam)\dbv_{a\in A}\langle a,y\rangle}{||\lam x+(1-\lam)y||}=\dfrac{\lam h_A(x)+(1-\lam)h_A(y)}{||\lam x+(1-\lam)y||}$$
for all $x,y\in S^{n-1}$, $\lam\in[0,1]$ with $\lam x+(1-\lam)y\neq 0$.

{\bf Sufficiency.} Let $h:S^{n-1}\to\bbR$ be a sublinear function. Since
\begin{equation} \label{Ah-intersection}
A_h=\{t\in\bbR^n\mid\forall x\in S^{n-1}:\ \langle t,x\rangle\leq h(x)\}=\bigcap_{x\in S^{n-1}}\{t\in\bbR^n\mid\langle t,x\rangle\leq h(x)\}
\end{equation}
is the intersection of closed halfspaces which are necessarily convex, it is clearly closed and convex. Moreover, considering the standard basis $\{e_1,\dots,e_n\}$ of $\bbR^n$ we see that each coordinate of $t\in A_h$ is bounded by
$$\al_0:=\max\{|h(e_i)|,|h(-e_i)|\mid 1\leq i\leq n\},$$
and thus $||t||\leq\sqrt{n}\al_0$ for all $t\in A_h$; that is, $A_h$ is bounded. Next, we show that $A_h\neq\varnothing$ and $h_{A_h}=h$, so that $A_h$ is a convex body and $h$ is the support function of $A_h$.

Firstly, $A_h\neq\varnothing$. To see this, let us extend $h:S^{n-1}\to\bbR$ to $\th:\bbR^n\to\bbR$ as elaborated in Remark \ref{cb-supp-remark}, i.e.,
$$\th(x)=\begin{cases}
0 & \text{if}\ x=o,\\
||x||\cdot h\left(\dfrac{x}{||x||}\right) & \text{else}
\end{cases}$$
for all $x\in\bbR^n$. Then
\begin{equation} \label{th-sublinear}
\th(\lam x)=\lam\th(x)\quad\text{and}\quad\th(x+y)\leq\th(x)+\th(y)
\end{equation}
for all $x,y\in\bbR^n$ and $\lam\geq 0$, where the first equality follows obviously from the definition of $\th$, while the second inequality holds because
$$\th(o)=0\leq ||x||\cdot\left(h\left(\dfrac{x}{||x||}\right)+h\left(-\dfrac{x}{||x||}\right)\right)=\th(x)+\th(-x)$$
for all $x\in\bbR^n\setminus\{o\}$, and
\begin{align*}
\th(x+y)&=||x+y||\cdot h\left(\dfrac{x+y}{||x+y||}\right)=||x+y||\cdot h\left(\dfrac{\dfrac{||x||}{||x||+||y||}\cdot
\dfrac{x}{||x||}+\dfrac{||y||}{||x||+||y||}\cdot\dfrac{y}{||y||}}{\left|\left|\dfrac{||x||}{||x||+||y||}\cdot\dfrac{x}{||x||}+\dfrac{||y||}{||x||+||y||}\cdot\dfrac{y}{||y||}\right|\right|}\right)\\
&\leq||x+y||\cdot\dfrac{\dfrac{||x||}{||x||+||y||}\cdot
h\left(\dfrac{x}{||x||}\right)+\dfrac{||y||}{||x||+||y||}\cdot h\left(\dfrac{y}{||y||}\right)}{\left|\left|\dfrac{||x||}{||x||+||y||}\cdot\dfrac{x}{||x||}+\dfrac{||y||}{||x||+||y||}\cdot\dfrac{y}{||y||}\right|\right|}=\th(x)+\th(y)
\end{align*}
for all $x,y\in\bbR^n$ with $x+y\neq o$.

Since $\th:\bbR^n\to\bbR$ is clearly convex by \eqref{th-sublinear}, Lemma \ref{f-convex-epif-convex} tells us that 
$$\epi\th=\{(x,\al)\in\bbR^n\times\bbR\mid\th(x)\leq\al\}\subseteq\bbR^{n+1}$$ 
is closed and convex. By Theorem \ref{convex-body-support-hyperplane}, there exists a hyperplane $H_y$ that supports $\epi\th$ at any $(y,\th(y))\in\bd\epi\th$ with $y\neq o$. We claim that $H_y$ passes through the origin $(o,0)$ of $\bbR^n\times\bbR=\bbR^{n+1}$, so that $H_y$ also supports $\epi\th$ at $(o,0)$. In fact, if $(o,0)\not\in H_y$, then the straight line passing through $(o,0)$ and $(y,\th(y))$ intersects the hyperplane $H_y$ in exactly one point, i.e., $(y,\th(y))$. It follows that $(o,0)$ and $(2y,2\th(y))$ are on different sides of $H_y$. However, it is clear that $(o,0)$ and $(2y,2\th(y))$ are both in $\epi\th$, which contradicts to the fact that $H_y$ is a support hyperplane of $\epi\th$.

Note that the definition of $\epi\th$ indicates that the exterior normal vector of $H_y$ can be written as the form $(t_y,-1)$, which intuitively means that the exterior normal vector of $H_y$ ``points below $\bbR^n$'' in $\bbR^n\times\bbR$. Since $H_y$ also supports $\epi\th$ at $(o,0)$, we obtain
\begin{equation} \label{Hy-ty}
H_y=\{(x,\al)\in\bbR^n\times\bbR\mid\langle(x,\al),(t_y,-1)\rangle=\langle(o,0),(t_y,-1)\rangle=0\}.
\end{equation}
Since $\epi\th\subseteq H_y^-$, for any $x\in\bbR^n$, considering $(x,\th(x))\in\epi\th$ we have
$$\langle t_y,x\rangle-\th(x)=\langle(x,\th(x)),(t_y,-1)\rangle\leq 0,$$
i.e., $\langle t_y,x\rangle\leq\th(x)$. In particular, this means that $\langle t_y,x\rangle\leq h(x)$ for all $x\in S^{n-1}$, and consequently $t_y\in A_h$ (cf. \eqref{Ah-intersection}), showing that $A_h\neq\varnothing$.

Secondly, $h_{A_h}=h$. On one hand,
$$h_{A_h}(x)=\bv_{a\in A_h}\langle a,x\rangle\leq h(x)$$
for all $x\in S^{n-1}$ is an immediate consequence of \eqref{hA-def} and \eqref{Ah-intersection}. On the other hand, for any $y\in S^{n-1}$, from $(y,h(y))\in H_y$ we have $\langle t_y,y\rangle=h(y)$ by \eqref{Hy-ty}. Since $t_y\in A_h$, it follows that
$$h_{A_h}(y)=\bv_{a\in A_h}\langle a,y\rangle\geq\langle t_y,y\rangle=h(y)$$
for all $y\in S^{n-1}$. Hence $h_{A_h}=h$.

Finally, for the uniqueness of $A_h$, it remains to show that $A_{h_A}=A$ for any convex body $A\in\CC^n$. Note that $A\subseteq A_{h_A}$ follows easily from \eqref{hA-def} and \eqref{Ah-intersection}. For the reverse inclusion, we proceed by contradiction. Suppose that $t_0\in A_{h_A}$ but $t_0\not\in A$. Since $A$ is closed and convex, by Lemma \ref{convex-body-support-hyperplane-not-in} we may find a hyperplane $H$ that supports $A$ at $p_A(t_0)$, with
$$x_0:=\dfrac{t_0-p_A(t_0)}{||t_0-p_A(t_0)||}\in S^{n-1}$$ 
being its exterior normal vector, i.e.,
$$H=\{t\in\bbR^n\mid\langle t,x_0\rangle=\langle p_A(t_0),x_0\rangle\}.$$ 
Now, it follows from $A\subseteq H^-$ that $\langle a,x_0\rangle\leq\langle p_A(t_0),x_0\rangle$ for all $a\in A$, but $t_0\in\bbR^n\setminus H^-$ forces $\langle t_0,x_0\rangle>\langle p_A(t_0),x_0\rangle$; that is,
$$\langle t_0,x_0\rangle>\langle p_A(t_0),x_0\rangle\geq\bv_{a\in A}\langle a,x_0\rangle=h_A(x_0),$$
contradicting to $t_0\in A_{h_A}$. This completes the proof.
\end{proof}

As an immediate consequence of Theorem \ref{cb-supp}, the following characterization of elements of convex bodies will be useful later (cf. \cite[Theorem 2.2]{Zhang2001}):

\begin{prop} \label{cb-elements}
Let $A\in\CC^n$ be a convex body. Then $t\in A$ if, and only if, $\langle t,x\rangle\leq h_A(x)$ for all $x\in S^{n-1}$.
\end{prop}

In particular, constant support functions correspond to closed balls centered at the origin:

\begin{prop} \label{cb-ball}
A convex body $A\in\CC^n$ is a closed ball centered at the origin if, and only if, its support function $h_A:S^{n-1}\to\bbR$ is a (nonnegative) constant function. In this case, $h_A$ is necessarily constant at the radius of $A$.
\end{prop}

\begin{proof}
Suppose that $A$ is a closed ball of radius $\lam$ centered at the origin. Then
\begin{equation} \label{hA=ga}
h_A(x)=\bv_{a\in A}\langle a,x\rangle=\langle\lam x,x\rangle=\lam
\end{equation}
for all $x\in S^{n-1}$. Conversely, if $h_A(x)=\lam$ for all $x\in S^{n-1}$, then it follows from Proposition \ref{cb-elements} that
$$t\in A\iff\forall x\in S^{n-1}:\ \langle t,x\rangle\leq\lam\iff||t||=\left\langle t,\dfrac{t}{||t||}\right\rangle\leq\lam,$$
which also guarantees the nonnegativity of $\lam$.
\end{proof}

Recall that the \emph{Minkowski sum} \cite{Schneider2013} of convex bodies $B,C\in\CC^n$ is given by
$$B+C:=\{b+c\mid b\in B,c\in C\}.$$
Note that a direct computation
$$h_{B+C}(x)=\bv_{b\in B,c\in C}\langle b+c,x\rangle=\bv_{b\in B,c\in C}(\langle b,x\rangle+\langle c,x\rangle)=\bv_{b\in B}\langle b,x\rangle+\bv_{c\in C}\langle c,x\rangle=h_B(x)+h_C(x)$$
for any $x\in S^{n-1}$ shows that $h_{B+C}=h_B+h_C$ (cf. \cite[Theorem 1.7.5]{Schneider2013}), in combination with Theorem \ref{cb-supp} we obtain:

\begin{prop} \label{cb-sum-supp}
For convex bodies $A,B,C\in\CC^n$, $A=B+C$ if, and only if, $h_A=h_B+h_C$.
\end{prop}

\subsection{Fuzzy vectors via convex bodies and support functions} \label{Vec-CB}

Following the terminology of \cite{Maeda2008}, by a \emph{fuzzy vector} we mean a fuzzy subset of $\bbR^n$, i.e., a function 
$$u:\bbR^n\to[0,1],$$
subject to the following requirements:
\begin{enumerate}[label=(V\arabic*)]
\item \label{vec:reg} $u$ is \emph{regular}, i.e., there exists $t_0\in\bbR^n$ with $u(t_0)=1$;
\item \label{vec:comp} $u$ is \emph{compactly supported}, i.e., the closure of $\{t\in\bbR^n\mid u(t)>0\}$ is compact;
\item \label{vec:conv} $u$ is \emph{quasi-concave}, i.e., $u(s)\wedge u(t)\leq u(\lam s+(1-\lam)t)$ for all $s,t\in\bbR^n$ and $\lam\in[0,1]$;
\item \label{vec:usc} $u$ is \emph{upper semi-continuous}, i.e., $\{t\in\bbR^n\mid u(t)\geq\al\}$ is closed for all $\al\in[0,1]$.
\end{enumerate}

The set of all fuzzy vectors of dimension $n$ is denoted by $\CF^n$, and a canonical embedding of $\bbR^n$ into $\CF^n$ assigns to each $a\in\bbR^n$ a ``crisp'' fuzzy vector
$$\tilde{a}:\bbR^n\to[0,1],\quad\tilde{a}(t):=\begin{cases}
1 & \text{if}\ t=a,\\
0 & \text{else}.
\end{cases}$$

\begin{rem}
The conditions \ref{vec:reg}--\ref{vec:usc}, first appeared in \cite{Kaleva1985} and \cite{Puri1985}, are originated from the definition of \emph{fuzzy numbers}, i.e., fuzzy vectors of dimension $1$ (see \cite{Dubois1978,Dubois1982,Dubois1982a,Dubois1982b,Goetschel1983}); so, fuzzy vectors are also called $n$-dimensional fuzzy numbers (see \cite{Zhang2001,Zhang2002b,Wang2002,Wang2007}). It should be reminded that \emph{$n$-dimensional fuzzy vectors} defined in \cite{Wang2007} are different from our fuzzy vectors here.
\end{rem}

For each $u\in\CF^n$ and $\al\in[0,1]$, the \emph{$\al$-level sets} of $u$ are defined as
\renewcommand\arraystretch{1.3}
$$u_{\al}:=\left\{\begin{array}{ll}
\{t\in\bbR^n\mid u(t)\geq\al\} & \text{if}\ \al\in(0,1],\\
\overline{\bigcup\limits_{\al\in(0,1]}u_{\al}}=\overline{\{t\in\bbR^n\mid u(t)>0\}} & \text{if}\ \al=0.
\end{array}
\right.$$
It is easy to see that
\begin{equation} \label{vec-level}
u(t)=\bv_{t\in u_{\al}}\al
\end{equation}
for each $u\in\CF^n$ and $t\in\bbR^n$.

\begin{rem} \label{al-empty-sup}
Since $\al$ ranges in the closed interval $[0,1]$, the supremum in Equation \eqref{vec-level} is computed in $[0,1]$. Hence, if $t\not\in u_{\al}$ for all $\al\in[0,1]$, i.e.,
$$\{\al\mid t\in u_{\al}\}=\varnothing,$$
then $u(t)=0$ because $0$, as the bottom element of $[0,1]$, is the supremum of the empty subset of $[0,1]$.
\end{rem}

It is well known that fuzzy vectors can be characterized via convex bodies as follows, whose proof is straightforward and will be omitted here:

\begin{thm} \label{vec-cb} (See \cite{Negoita1975,Kaleva1987}.)
Let $\{A_{\al}\mid\al\in[0,1]\}$ be a family of subsets of $\bbR^n$. Then there exists a (unique) fuzzy vector
$$u:\bbR^n\to[0,1],\quad u(t)=\bv_{t\in A_{\al}}\al$$
such that
$$u_{\al}=A_{\al}$$
for all $\al\in[0,1]$ if, and only if,
%$$u_{\al}=\left\{\begin{array}{ll}
%A_{\al} & \text{if}\ \al\in(0,1],\\
%\overline{\bigcup\limits_{\al\in(0,1]}A_{\al}}\subseteq A_0 & \text{if}\ \al=0
%\end{array}\right.$$
%if, and only if,
\begin{enumerate}[label={\rm(L\arabic*)}]
\item \label{vec-cb:convex} $A_{\al}$ is a convex body in $\bbR^n$ for each $\al\in[0,1]$;
\item \label{vec-cb:monotone} $A_{\al}\supseteq A_{\be}$ whenever $0\leq\al<\be\leq 1$;
\item \label{vec-cb:cont} $A_{\al_0}=\bigcap\limits_{k\geq 1}A_{\al_k}$ for each increasing sequence $\{\al_k\}\subseteq[0,1]$ that converges to $\al_0>0$;
\item \label{vec-cb:0} $A_0=\overline{\bigcup\limits_{\al\in(0,1]}A_{\al}}$.
\end{enumerate}
\end{thm}

Since all the level sets of a fuzzy vector are convex bodies, it makes sense to define the \emph{support function} \cite{Puri1985,Diamond1989} of $u\in\CF^n$ as
$$h_u:[0,1]\times S^{n-1}\to\bbR,\quad h_u(\al,x):=\bv_{t\in u_{\al}}\langle t,x\rangle;$$
that is, $h_u(\al,-):=h_{u_{\al}}$ is the support function of the convex body $u_{\al}$ for each $\al\in[0,1]$. In particular, $h_u$ is bounded on $[0,1]\times S^{n-1}$, because
$$h_u(\al,x)\leq\bv_{t\in u_0}\langle t,x\rangle=h_{u_0}(x)$$
for all $\al\in[0,1]$, $x\in S^{n-1}$, and $h_{u_0}$ is bounded on $S^{n-1}$.

With Theorems \ref{cb-supp} and \ref{vec-cb} we may describe fuzzy vectors through support functions as the following representation theorem reveals, which is the main result of this paper:

\begin{thm} \label{vec-supp}
A function $h:[0,1]\times S^{n-1}\to\bbR$ is the support function of a (unique) fuzzy vector
$$u:\bbR^n\to[0,1]$$
given by
$$u(t)=\bv_{t\in A_{h(\al,-)}}\al=\bv\{\al\mid\forall x\in S^{n-1}:\ \langle t,x\rangle\leq h(\al,x)\}$$
if, and only if,
\begin{enumerate}[label={\rm(VS\arabic*)}]
\item \label{vec-supp:convex} $h(\al,-):S^{n-1}\to\bbR$ is sublinear for each $\al\in[0,1]$, i.e.,
    $$h(\al,x)+h(\al,-x)\geq 0$$
    and
    $$h\left(\al,\dfrac{\lam x+(1-\lam)y}{||\lam x+(1-\lam)y||}\right)\leq\dfrac{\lam h(\al,x)+(1-\lam)h(\al,y)}{||\lam x+(1-\lam)y||}$$
    for all $\al\in[0,1]$, $x,y\in S^{n-1}$, $\lam\in[0,1]$ with $\lam x+(1-\lam)y\neq 0$,
\item \label{vec-supp:x} $h(-,x):[0,1]\to\bbR$ is non-increasing, left-continuous on $(0,1]$ and right-continuous at $0$ for each $x\in S^{n-1}$.
%$h(\be,x)\leq h(\al,x)$ whenever $0\leq\al<\be\leq 1$ and $x\in S^{n-1}$;
%\item \label{vec-supp:convex:cont} $h(-,x):[0,1]\to\bbR$ is continuous for all $x\in S^{n-1}$.
\end{enumerate}
%Let $u\in\CF^n$. Then
%\begin{enumerate}[label={\rm(VS\arabic*)}]
%\item \label{vec-supp:ub}
%    $h_u$ is uniformly bounded on $[0,1]\times S^{n-1}$;
%\item \label{vec-supp:img}
%    the image of $h_u$ is compact;
%\item \label{vec-supp:Lips}
%    $h_u$ is continuous on $[0,1]\times S^{n-1}$ and, in particular, $h_u(\al,-)$ is Lipschitz continuous for each $\al\in[0,1]$;
%\item \label{vec-supp:subadd}
%    $h_u(\al,-)$ is sublinear for each $\al\in[0,1]$.
%\end{enumerate}
%Conversely, for any function $h:[0,1]\times S^{n-1}\to\bbR$ satisfying the conditions \ref{vec-supp:ub}-\ref{vec-supp:subadd}, there exists a unique $u\in\CF^n$ such that
%$$\begin{cases}
%u_{\al}=\{t\in\bbR^n\mid\forall x\in S^{n-1}:\ \langle t,x\rangle\leq h(\al,x)\} & \text{if}\ \al\in(0,1],\\
%u_0=\overline{\bigcup\limits_{\al\in(0,1]}u_{\al}} & \text{if}\ \al=0.
%\end{cases}$$
\end{thm}

\begin{proof}
{\bf Necessity.} Let $u\in\CF^n$ be a fuzzy vector. Then $h_u$ clearly satisfies \ref{vec-supp:convex} by Theorem \ref{cb-supp}. For \ref{vec-supp:x}, let us fix $x\in S^{n-1}$. Then $h_u(-,x):[0,1]\to\bbR$ is non-increasing because of \ref{vec-cb:monotone}.

To see that $h_u(-,x)$ is left-continuous at each $\al_0\in(0,1]$, let $\{\al_k\}\subseteq(0,1]$ be an increasing sequence that converges to $\al_0$. Then $u_{\al_0}=\bigcap\limits_{k\geq 1}u_{\al_k}$ by \ref{vec-cb:cont}. For each $\ep>0$, we claim that there exists a positive integer $k$ such that for all $t\in u_{\al_k}$, there exists $r_t\in u_{\al_0}$ with 
\begin{equation} \label{rt-def}
||t-r_t||<\ep.
\end{equation} 
Indeed, suppose that we find an $\ep_0>0$ such that for all positive integers $k$, there exists $t_k\in u_{\al_k}$ such that
$$d(t_k,u_{\al_0}):=\bw_{r\in u_{\al_0}}||t_k-r||\geq\ep_0.$$
Then the sequence $\{t_k\}$ is contained in the compact set $u_{\al_1}$, and thus it has a convergent subsequence. Without loss of generality we may suppose that $\lim\limits_{k\ra\infty}t_k=t_0$. Then $t_0\in u_{\al_k}$ for all $k\geq 1$ because $\{t_m\mid m\geq k\}\subseteq u_{\al_k}$, and consequently $t_0\in \bigcap\limits_{k\geq 1}u_{\al_k}=u_{\al_0}$. But the construction of $t_k$ forces $d(t_0,u_{\al_0})\geq\ep_0$, which is a contradiction.

Note that \eqref{rt-def} actually means that
$$u_{\al_k}\subseteq u_{\al_0}+B_{\ep},$$
where $B_{\ep}$ refers to the closed ball of radius $\ep$ centered at the origin, since each $t\in u_{\al_k}$ may be written as
$$t=r_t+(t-r_t)$$
with $r_t\in u_{\al_0}$ and $t-r_t\in B_{\ep}$. Then it follows from Propositions \ref{cb-ball} and \ref{cb-sum-supp} that
$$h_u(\al_k,x)=h_{u_{\al_k}}(x)\leq h_{u_{\al_0}}(x)+h_{B_{\ep}}(x)=h_u(\al_0,x)+\ep.$$
%Suppose that $r_t=t$ for $t\in u_{\al_0}$ (because $r_t=t$ obviously satisfies \eqref{rt-def} when $t\in u_{\al_0}$). Then
%\begin{equation} \label{rt-x-t-x}
%\bv_{t\in u_{\al_k}}\langle r_t,x\rangle=\bv_{t\in u_{\al_0}}\langle t,x\rangle.
%\end{equation}
%Indeed, $\dbv_{t\in u_{\al_k}}\langle r_t,x\rangle\leq\dbv_{t\in u_{\al_0}}\langle t,x\rangle$ because $r_t\in u_{\al_0}$ for all $t\in u_{\al_k}$, and from $u_{\al_0}\subseteq u_{\al_k}$ in combination with $r_t=t$ for all $t\in u_{\al_0}$ we see that $\dbv_{t\in u_{\al_0}}\langle t,x\rangle=\dbv_{t\in u_{\al_0}}\langle r_t,x\rangle\leq\dbv_{t\in u_{\al_k}}\langle r_t,x\rangle$. It follows that
%\begin{align*}
%0&\leq h_u(\al_k,x)-h_u(\al_0,x)&(h_u(-,x)\ \text{is non-increasing})\\
%&=\bv_{t\in u_{\al_k}}\langle t,x\rangle-\bv_{t\in u_{\al_0}}\langle t,x\rangle\\
%&\leq\bv_{t\in u_{\al_k}}\langle t-r_t,x\rangle+\bv_{t\in u_{\al_k}}\langle r_t,x\rangle-\bv_{t\in u_{\al_0}}\langle t,x\rangle\\
%&=\bv_{t\in u_{\al_k}}\langle t-r_t,x\rangle&(\text{Equation \eqref{rt-x-t-x}})\\
%&\leq\bv_{t\in u_{\al_k}}||x||\cdot||t-r_t||\\
%&\leq\ep.&(x\in S^{n-1}\ \text{and \eqref{rt-def}})
%\end{align*}
Hence, together with the monotonicity of $h_u(-,x):[0,1]\to\bbR$ we conclude that $\lim\limits_{k\ra\infty}h_u(\al_k,x)=h_u(\al_0,x)$, which proves the left-continuity of $h_u(-,x)$ at $\al_0$.

To see that $h_u(-,x)$ is right-continuous at $0$, let $\ep>0$. The compactness of $u_0$ allows us to find $q_1,\dots,q_k\in u_0$ such that $u_0$ is covered by finitely many open balls
$$B\left(q_1,\dfrac{\ep}{2}\right),\dots,B\left(q_k,\dfrac{\ep}{2}\right)$$
centered at $q_1,\dots,q_k$, respectively, with radii $\dfrac{\ep}{2}$.

Note that for each $t\in u_0=\overline{\bigcup\limits_{\al\in(0,1]}u_{\al}}$, there exists $\al_t\in(0,1]$ and $s_t\in u_{\al_t}$ such that $||t-s_t||<\dfrac{\ep}{2}$. Let
$$\al_q:=\min\{\al_{q_1},\dots,\al_{q_k}\}>0,$$
and let $B\left(q_t,\dfrac{\ep}{2}\right)$ $(q_t\in\{q_1,\dots,q_k\})$ be the open ball containing $t$. Then $s_{q_t}\in u_{\al_{q_t}}\subseteq u_{\al_q}$, and
$$||t-s_{q_t}||\leq ||t-q_t||+||q_t-s_{q_t}||<\dfrac{\ep}{2}+\dfrac{\ep}{2}=\ep.$$
It follows that
$$u_0\subseteq u_{\al_q}+B_{\ep},$$
because $t=s_{q_t}+(t-s_{q_t})$ for all $t\in u_0$, where $s_{q_t}\in u_{\al_q}$ and $t-s_{q_t}\in B_{\ep}$. By Propositions \ref{cb-ball} and \ref{cb-sum-supp} this means that
$$h_u(0,x)=h_{u_0}(x)\leq h_{u_{\al_q}}(x)+h_{B_{\ep}}(x)=h_u(\al_q,x)+\ep.$$
%Suppose that $s_{q_t}=t$ for all $t\in u_{\al_q}$, we obtain
%\begin{align*}
%0&\leq h_u(0,x)-h_u(\al_q,x)&(h_u(-,x)\ \text{is non-increasing})\\
%&=\bv_{t\in u_0}\langle t,x\rangle-\bv_{t\in u_{\al_q}}\langle t,x\rangle\\
%&\leq\bv_{t\in u_0}\langle t-s_{q_t},x\rangle+\bv_{t\in u_0}\langle s_{q_t},x\rangle-\bv_{t\in u_{\al_q}}\langle t,x\rangle\\
%&=\bv_{t\in u_0}\langle t-s_{q_t},x\rangle\\
%&\leq\bv_{t\in u_0}||x||\cdot||t-s_{q_t}||\\
%&\leq\ep.
%\end{align*}
Hence, together with the monotonicity of $h_u(-,x):[0,1]\to\bbR$ we conclude that $\lim\limits_{\al\ra 0+}h_u(\al,x)=h_u(0,x)$, which proves the right-continuity of $h_u(-,x)$ at $0$.

{\bf Sufficiency.} It suffices to show that $\{A_{h(\al,-)}\mid\al\in[0,1]\}$ satisfies the conditions \ref{vec-cb:convex}--\ref{vec-cb:0}.

Firstly, \ref{vec-cb:convex} is a direct consequence of Theorem \ref{cb-supp}.

Secondly, \ref{vec-cb:monotone} holds since $h(-,x)$ is non-increasing for all $x\in S^{n-1}$.

Thirdly, for \ref{vec-cb:cont}, let $\{\al_k\}\subseteq(0,1]$ be an increasing sequence that converges to $\al_0\in(0,1]$. It remains to show that
$$\bigcap\limits_{k\geq 1}A_{h(\al_k,-)}\subseteq A_{h(\al_0,-)},$$
since the reverse inclusion is trivial by \ref{vec-cb:monotone}. Suppose that $t\in A_{h(\al_k,-)}$ for all $k\geq 1$. Then, by Proposition \ref{cb-elements}, $\langle t,x\rangle\leq h(\al_k,x)$ for all $x\in S^{n-1}$. Thus the left-continuity of $h(-,x)$ at $\al_0$ implies that
$$\langle t,x\rangle\leq\lim\limits_{k\ra\infty}h(\al_k,x)=h\left(\lim\limits_{k\ra\infty}\al_k,x\right)=h(\al_0,x),$$
and consequently $t\in A_{h(\al_0,-)}$.

Finally, we prove \ref{vec-cb:0} by showing that
$$A_{h(0,-)}\subseteq A:=\overline{\bigcup\limits_{\al\in(0,1]}A_{h(\al,-)}}$$
as the reverse inclusion is trivial by \ref{vec-cb:monotone}. We proceed by contradiction. Suppose that $t_0\in A_{h(0,-)}$ but $t_0\not\in A$. Note that $A$ is also a convex body, since $A\subseteq A_{h(0,-)}$ and $A$ is the closure of the union of a family of convex bodies linearly ordered by inclusion. Thus, by Lemma \ref{convex-body-support-hyperplane-not-in} we may find a hyperplane $H$ that supports $A$ at $p_A(t_0)$, with
$$x_0:=\dfrac{t_0-p_A(t_0)}{||t_0-p_A(t_0)||}\in S^{n-1}$$
as its exterior normal vector, i.e.,
$$H=\{t\in\bbR^n\mid\langle t,x_0\rangle=\langle p_A(t_0),x_0\rangle\},$$  
which necessarily satisfies $t_0\in\bbR^n\setminus H^-$ and $A\subseteq H^-$. It follows that
\begin{align*}
h(\al,x_0)&=\bv_{t\in A_{h(\al,-)}}\langle t,x_0\rangle&(\text{Theorem \ref{cb-supp}})\\
&\leq\langle p_A(t_0),x_0\rangle&(A_{h(\al,-)}\subseteq A\subseteq H^-)\\
&=\langle t_0,x_0\rangle-\langle t_0-p_A(t_0),x_0\rangle\\
&=\langle t_0,x_0\rangle-||t_0-p_A(t_0)||&\left(x_0=\dfrac{t_0-p_A(t_0)}{||t_0-p_A(t_0)||}\right)\\
&\leq h(0,x_0)-||t_0-p_A(t_0)||&(t_0\in A_{h(0,-)})
\end{align*}
for all $\al\in(0,1]$, which contradicts to the right-continuity of $h(-,x_0)$ at $0$. The proof is thus completed.
\end{proof}

\section{Mare{\v s} cores of fuzzy vectors} \label{Mares}

\subsection{Addition of fuzzy vectors} \label{Vec-Add}
%Addition $\oplus$ and substraction $\ominus$ of fuzzy vectors $u,v\in\CF^n$ are defined by Zadeh's extension principle (cf. \cite{Dubois1978,Zadeh1975}), i.e.,
With the results of Section \ref{Representation} we are now able to characterize the addition $\oplus$ of fuzzy vectors through their level sets and support functions. Explicitly, the sum
$$u\oplus v\in\CF^n$$
of fuzzy vectors $u,v\in\CF^n$ is defined by Zadeh's extension principle (cf. \cite{Zadeh1975,Dubois1978}), i.e.,
%\begin{align*}
%&(u\oplus v)(t):=\bv_{r+s=t} u(r)\wedge v(s),\\
%&(u\ominus v)(t):=\bv_{r-s=t}u(r)\wedge v(s)
%\end{align*}
%Moreover, the scalar multiplication of $u\in\CF^n$ by $k\in\bbR$ is defined by
%$$ku:\bbR^n\to[0,1],\quad (ku)(t):=\begin{cases}
%u\left(\dfrac{t}{k}\right) & \text{if}\ k\neq 0,\\
%0 & \text{if}\ k=0.
%\end{cases}$$
$$(u\oplus v)(t):=\bv_{r+s=t} u(r)\wedge v(s)$$
for all $t\in\bbR^n$.

%\begin{proof}
%With Theorem \ref{cb-supp} it suffices to prove that $h_{B+C}=h_B+h_C$, which follows from
%\begin{align*}
%h_{B+C}(x)&=\bv_{a\in B+C}\langle a,x\rangle=\bv_{b\in B,c\in C}\langle b+c,x\rangle\\
%&=\bv_{b\in B,c\in C}(\langle b,x\rangle+\langle c,x\rangle)=\bv_{b\in B}\langle b,x\rangle+\bv_{c\in C}\langle c,x\rangle\\
%&=h_B(x)+h_C(x)
%\end{align*}
%for all $x\in S^{n-1}$.
%\end{proof}

\begin{thm} \label{vec-sum-supp}
For fuzzy vectors $u,v,w\in\CF^n$, the following statements are equivalent:
\begin{enumerate}[label={\rm(\roman*)}]
\item \label{vec-sum-supp:vec} $u=v\oplus w$.
\item \label{vec-sum-supp:level} $u_{\al}=v_{\al}+w_{\al}$ for all $\al\in[0,1]$.
\item \label{vec-sum-supp:supp} $h_u=h_v+h_w$.
\end{enumerate}
\end{thm}

\begin{proof}
\ref{vec-sum-supp:level}$\iff$\ref{vec-sum-supp:supp} is an immediate consequence of Proposition \ref{cb-sum-supp}. For \ref{vec-sum-supp:vec}$\iff$\ref{vec-sum-supp:level}, by Theorem \ref{vec-cb} it suffices to observe that 
\begin{equation} \label{v-oplus-w}
(v\oplus w)_{\al}=v_{\al}+w_{\al}
\end{equation} 
for all $\al\in[0,1]$, which is a well-known fact of Zadeh's extension principle (see, e.g., \cite[Proposition 3.3]{Nguyen1978}). For the sake of self-containment we give a proof of \eqref{v-oplus-w} here. On one hand, $v_{\al}+w_{\al}\subseteq(v\oplus w)_{\al}$. Let $r\in v_{\al}$, $s\in w_{\al}$. Then $v(r)\geq\al$ and $w(s)\geq\al$, and consequently
$$(v\oplus w)(r+s)\geq v(r)\wedge w(s)\geq\al,$$
showing that $r+s\in(v\oplus w)_{\al}$. On the other hand, in order to verify the reverse inclusion, let $t\in(v\oplus w)_{\al}$. Then
$$(v\oplus w)(t)=\bv_{r+s=t} v(r)\wedge w(s)\geq\al.$$
Consequently, there exists a sequence $\{r_k\}$ in $\bbR^n$ such that
$$v(r_k)\wedge w(t-r_k)\geq\left(1-\dfrac{1}{2^k}\right)\al\geq\dfrac{\al}{2}$$
for any positive integer $k$. In particular, this means that $\{r_k\}\subseteq u_{\frac{\al}{2}}$. Since $u_{\frac{\al}{2}}$ is compact, $\{r_k\}$ has a convergent subsequence, and without loss of generality we may assume that $\lim\limits_{k\ra\infty}r_k=r_0$. Note that for any positive integers $k,l$ with $l\geq k$, 
$$v(r_l)\wedge w(t-r_l)\geq\left(1-\dfrac{1}{2^l}\right)\al\geq\left(1-\dfrac{1}{2^k}\right)\al.$$
Letting $l\ra\infty$ in the above inequality, the upper semi-continuity of $v,w$ (see \ref{vec:usc}) then implies that
$$v(r_0)\wedge w(t-r_0)\geq\left(1-\dfrac{1}{2^k}\right)\al.$$
Since $k$ is arbitrary, the above inequality forces $v(r_0)\wedge w(t-r_0)\geq\al$. It follows that $r_0\in v_{\al}$ and $t-r_0\in w_{\al}$, and therefore $t=r_0+(t-r_0)\in v_{\al}+w_{\al}$.
\end{proof}

\subsection{Symmetric fuzzy vectors} \label{Vec-Sym}

Let $O(n)$ denote the orthogonal group of dimension $n$, i.e., the group of $n\times n$ orthogonal matrices.

\begin{defn} \label{sym-def}
A fuzzy vector $u\in\CF^n$ is \emph{symmetric (around the origin)} if it is $O(n)$-invariant; that is, if
$$u(t)=u(Qt)$$
for all $t\in\bbR^n$ and $Q\in O(n)$.
\end{defn}

We denote by $\Fs$ the set of all symmetric fuzzy vectors of dimension $n$.

\begin{rem}
In the case of $n=1$, since $O(1)=\{-1,1\}$, $u\in\CF_1$ is symmetric if $u(t)=u(-t)$ for all $t\in\bbR$; that is, $u$ is a symmetric fuzzy number in the sense of Mare{\v s} \cite{Marevs1989,Marevs1992}. Hence, the symmetry of fuzzy numbers is a special case of Definition \ref{sym-def}. In fact, as indicated by \cite[Remark 2.1]{Qiu2014}, a symmetric fuzzy number actually refers to a fuzzy number that is \emph{symmetric around zero}.
\end{rem}

\begin{thm} \label{sym-vec-supp}
For each fuzzy vector $u\in\CF^n$, the following statements are equivalent:
\begin{enumerate}[label={\rm(\roman*)}]
\item \label{sym-vec-supp:vec} $u$ is symmetric.
\item \label{sym-vec-supp:rot} For each $\al\in[0,1]$, $u_{\al}$ is invariant under the action of $O(n)$; that is, $Qt\in u_{\al}$ for all $t\in u_{\al}$ and $Q\in O(n)$.
\item \label{sym-vec-supp:level} For each $\al\in[0,1]$, $u_{\al}$ is a closed ball centered at the origin.
\item \label{sym-vec-supp:supp} For each $\al\in[0,1]$, $h_u(\al,-):S^{n-1}\to\bbR$ is a (nonnegative) constant function.
\end{enumerate}
\end{thm}

\begin{proof}
Since \ref{sym-vec-supp:level}$\iff$\ref{sym-vec-supp:supp} is an immediate consequence of Proposition \ref{cb-ball}, it remains to prove that \ref{sym-vec-supp:vec}$\implies$\ref{sym-vec-supp:supp} and \ref{sym-vec-supp:level}$\implies$\ref{sym-vec-supp:rot}$\implies$\ref{sym-vec-supp:vec}.

\ref{sym-vec-supp:vec}$\implies$\ref{sym-vec-supp:supp}: Let $\al\in(0,1]$ and $x\in S^{n-1}$. For each $Q\in O(n)$ and $t\in u_{\al}$, the $O(n)$-invariance of $u$ implies that $Q^{-1}t\in u_{\al}$ since $u(Q^{-1}t)=u(t)\geq\al$, and consequently
$$\langle t,Qx\rangle=\langle Q^{-1}t,x\rangle\leq h_u(\al,x).$$
Thus
$$h_u(\al,Qx)=\bv_{t\in u_{\al}}\langle t,Qx\rangle\leq h_u(\al,x).$$
Since $Q$ is arbitrary, it also holds that $h_u(\al,Q^{-1}x)\leq h_u(\al,x)$, and consequently $h_u(\al,x)\leq h_u(\al,Qx)$. Hence
$$h_u(\al,Qx)=h_u(\al,x)$$
for all $Q\in O(n)$. Note that the function
$$O(n)\to S^{n-1},\quad Q\mapsto Qx$$
is surjective, and thus
$$h_u(\al,x)=h_u(\al,y)$$
for all $x,y\in S^{n-1}$; that is, $h_u(\al,-):S^{n-1}\to\bbR$ is constant.

In this case, in order to see that the value of $h(\al,-)$ is nonnegative, just note that for any $t\in u_{\al}$ with $t\neq o$, from $\dfrac{t}{||t||}\in S^{n-1}$ we deduce that
$$h_u(\al,-)=h_u\left(\al,\dfrac{t}{||t||}\right)\geq\left\langle t,\dfrac{t}{||t||}\right\rangle=||t||\geq 0.$$

\ref{sym-vec-supp:level}$\implies$\ref{sym-vec-supp:rot}: Let $\al\in[0,1]$. If $u_{\al}$ is a closed ball of radius $\lam_{\al}$ centered at the origin, then $Qt\in u_{\al}$ whenever $t\in u_{\al}$,  since $||t||\leq\lam_{\al}$ obviously implies that $||Qt||\leq\lam_{\al}$.

\ref{sym-vec-supp:rot}$\implies$\ref{sym-vec-supp:vec}: Let $t\in\bbR^n$ and $Q\in O(n)$. If $u(t)>0$, then $t\in u_{u(t)}$, and consequently $Qt\in u_{u(t)}$, i.e., $u(Qt)\geq u(t)$. As $Q$ is arbitrary, from $u(Q^{-1}t)\geq u(t)$ we immediately deduce that $u(t)\geq u(Qt)$, and thus $u(t)=u(Qt)$.

If $u(t)=0$, then $t\not\in u_{\al}$ for all $\al\in(0,1]$, and consequently $Qt\not\in u_{\al}$ for all $\al\in(0,1]$, which forces $u(Qt)=0$ and completes the proof.
\end{proof}

%Theorem \ref{sym-vec-supp} in combination with Theorem \ref{vec-sum-supp} and Corollary \ref{vec-minus-supp} leads to the following formula:
%
%\begin{cor} \label{vec-sym-sum}
%For $u,v\in\CF^n$ with $v\in\Fs$,
%$$h_{u\oplus v}(\al,-)=h_{u\ominus v}(\al,-)=h_u(\al,-)+h_v(\al)$$
%for all $\al\in[0,1]$.
%\end{cor}

From Theorem \ref{sym-vec-supp} we see that the support function $h_u:[0,1]\times S^{n-1}\to\bbR$ of a symmetric fuzzy vector $u\in\Fs$  actually reduces to a single-variable function
$$h_u:[0,1]\to\bbR,$$
and conversely:

\begin{cor} \label{function-sym-sup}
A function $h:[0,1]\to\bbR$ is the support function of a symmetric fuzzy vector $u\in\Fs$ if, and only if, $h$ is nonnegative, non-increasing, left-continuous on $(0,1]$ and right-continuous at $0$.
\end{cor}

\begin{proof}
Follows immediately from Theorems \ref{vec-supp} and \ref{sym-vec-supp}.
\end{proof}

As a direct application of Proposition \ref{cb-ball} and Theorem \ref{sym-vec-supp}, let us point out the following easy but useful facts:

\begin{cor} \label{sym-vec-sum}
Let $A,B\in\CC^n$ be convex bodies and $u,v\in\CF^n$ be fuzzy vectors.
\begin{enumerate}[label={\rm(\roman*)}]
\item \label{sym-vec-sum:ball} If $A$ and $B$ are both closed balls centered at the origin, then so is $A+B$, and
    $$\lam_{A+B}=\lam_A+\lam_B,$$
    where $\lam_A$, $\lam_B$ and $\lam_{A+B}$ are the radii of $A$, $B$ and $A+B$, respectively.
\item \label{sym-vec-sum:vec} If $u$ and $v$ are both symmetric, then so is $u\oplus v$.
\end{enumerate}
\end{cor}

\begin{proof}
For \ref{sym-vec-sum:ball}, just note that the support function of $A+B$ satisfies
$$h_{A+B}=h_A+h_B$$
by Proposition \ref{cb-sum-supp}, and thus $h_{A+B}$ is a (nonnegative) constant function since so are $h_A$ and $h_B$. The conclusion then follows from Proposition \ref{cb-ball}.

For \ref{sym-vec-sum:vec}, since $u$ and $v$ are both symmetric, for each $\al\in[0,1]$, Theorem \ref{sym-vec-supp} ensures that $h_u(\al,-)$ and $h_v(\al,-)$ are both constant on $S^{n-1}$ which, in conjunction with Theorem \ref{vec-sum-supp}, implies that
$$h_{u\oplus v}(\al,-)=h_u(\al,-)+h_v(\al,-)$$
is constant on $S^{n-1}$; that is, $u\oplus v$ is also symmetric.
\end{proof}

\subsection{Inner parallel bodies of convex bodies} \label{IPB}

Let $B_{\lam}$ denote the closed ball of radius $\lam\geq 0$ centered at the origin. Recall that the \emph{inner parallel body} (see \cite{Gruber2007,Schneider2013}) of a convex body $A\in\CC^n$ at distance $\lam$ is given by
$$A_{-\lam}=\{t\in\bbR^n\mid t+B_{\lam}\subseteq A\},$$
which is also a convex body as long as $A_{-\lam}\neq\varnothing$.

\begin{lem} \label{cb-pb-supp}
For convex bodies $A,B\in\CC^n$ and $\lam\geq 0$,
$$A=B+B_{\lam}\iff h_A=h_B+\lam\implies B=A_{-\lam}.$$
\end{lem}

\begin{proof}
The equivalence of $A=B+B_{\lam}$ and $h_A=h_B+\lam$ follows immediately from Propositions \ref{cb-ball} and \ref{cb-sum-supp}. In this case, from $A=B+B_{\lam}$ and the definition of $A_{-\lam}$ we soon see that $B\subseteq A_{-\lam}$. For the reverse inclusion, suppose that $a\in A_{-\lam}$. For each $x\in S^{n-1}$, note that $\lam x\in B_{\lam}$, and consequently
\begin{equation} \label{a-lamx-in-A}
a+\lam x\in A.
\end{equation}
It follows that
$$\langle a,x\rangle+\lam=\langle a,x\rangle+\langle\lam x,x\rangle=\langle a+\lam x,x\rangle\leq h_A(x),$$
where the last inequality is obtained by applying Proposition \ref{cb-elements} to \eqref{a-lamx-in-A}. Therefore,
$$\langle a,x\rangle\leq h_A(x)-\lam=h_B(x)$$
for all $x\in S^{n-1}$, and thus Proposition \ref{cb-elements} guarantees that $a\in B$.
%we proceed by contradiction. Suppose that $a\in A_{-\lam}$ but $a\not\in B$. Since $B$ is compact, there exists $b_0\in B$ such that
%\begin{equation} \label{daB}
%||a-b_0||=d(a,B)>0,
%\end{equation}
%where $d(a,B):=\bw\limits_{b\in B}||a-b||$ is the distance between $a$ and $B$. Let $\prod$ be the hyperplane passing through $b_0$ with $x_0:=\dfrac{a-b_0}{||a-b_0||}$ as its normal vector. Then $\prod$ divides $\bbR^n$ into two parts, say, $P_-$ and $P_+$, with
%\begin{align*}
%P_-&=\{t\in\bbR^n\mid\langle t,x_0\rangle\leq\langle b_0,x_0\rangle\}\quad\text{and}\\
%P_+&=\{t\in\bbR^n\mid\langle t,x_0\rangle\geq\langle b_0,x_0\rangle\}.
%\end{align*}
%Note that $a\in P_+$, and $B\subseteq P_-$ since $B$ is convex. Thus
%\begin{equation} \label{hBx0}
%h_B(x_0)=\bv_{b\in B}\langle b,x_0\rangle=\langle b_0,x_0\rangle.
%\end{equation}
%It follows that
%\begin{align*}
%&\langle a+\lam x_0,x_0\rangle\\
%={}&\langle a,x_0\rangle+\lam&(||x_0||=1)\\
%={}&\langle a-b_0,x_0\rangle+\langle b_0,x_0\rangle+\lam\\
%={}&||a-b_0||+h_B(x_0)+\lam&(\text{Equation \eqref{hBx0}})\\
%={}&||a-b_0||+h_A(x_0)&(h_A=h_B+\lam)\\
%>{}&h_A(x_0).&(\text{Inequality \eqref{daB}})
%\end{align*}
%Hence, from the definition of $h_A$ we see that $a+\lam x_0\not\in A$. However, since $a\in A_{-\lam}$ and $\lam x_0\in B_{\lam}$, according to the definition of $A_{-\lam}$ we must have $a+\lam x_0\in A$, which is a contradiction.
\end{proof}

In general, the last implication of Lemma \ref{cb-pb-supp} is proper; that is, $B=A_{-\lam}$ does not imply $A=B+B_{\lam}$. For example, let $A$ and $B$ be the hypercubes of side lengths $4$ and $3$, respectively, both centered at the origin. Then $B=A_{-1}$, but $A\supsetneq B+B_1$.

We say that a nonempty inner parallel body $A_{-\lam}$ of $A\in\CC^n$ is \emph{regular} if
$$A=A_{-\lam}+B_{\lam},$$
or equivalently (see Lemma \ref{cb-pb-supp}), if
$$h_A=h_{A_{-\lam}}+\lam.$$
For each convex body $A\in\CC^n$, we write
$$\Lam_A:=\{\lam\geq 0\mid A_{-\lam}\ \text{is a regular inner parallel body of}\ A\}.$$

\begin{prop} \label{LamA-closed-interval}
$\Lam_A$ is a closed interval, given by
$$\Lam_A=[0,\lam_A],$$
where $\lam_A:=\bv\Lam_A$.
\end{prop}

\begin{proof}
{\bf Step 1.} $A_{-\lam_A}$ is a regular inner parallel body of $A$.

In order to obtain $A_{-\lam_A}+B_{\lam_A}=A$, it suffices to show that every $t\in A$ lies in $A_{-\lam_A}+B_{\lam_A}$. Since $A=A_{-\lam}+B_{\lam}$ for all $\lam\in\Lam_A$, for an increasing sequence $\{\lam_k\}\subseteq\Lam_A$ that converges to $\lam_A$ we may find $a_k\in A_{-\lam_k}$ and $b_k\in B_{\lam_k}$ with $$t=a_k+b_k$$
for all positive integers $k$. Note that both the sequences $\{a_k\}$ and $\{b_k\}$ are bounded, and thus they have convergent subsequences. Without loss of generality we assume that $\lim\limits_{k\ra\infty}a_k=a_0$ and $\lim\limits_{k\ra\infty}b_k=b_0$. Then it is clear that
$$t=a_0+b_0$$
and $b_0\in B_{\lam_A}$, and it remains to prove that $a_0\in A_{-\lam_A}$. To this end, we need to show that $a_0+y\in A$ for all $y\in B_{\lam_A}$. Indeed, let $\{y_k\}\subseteq B_{\lam_A}$ be a sequence with $\lim\limits_{k\ra\infty}y_k=y$ and $y_k\in B_{\lam_k}$ for all positive integers $k$. Then from $a_k\in A_{-\lam_k}$ we deduce that $a_k+y_k\in A$, and consequently $a_0+y\in A$, as desired.

{\bf Step 2.} If $\lam\in\Lam_A$ and $0\leq\lam'<\lam$, then $\lam'\in\Lam_A$.

In order to obtain $A_{-\lam'}+B_{\lam'}=A$, it suffices to show that every $t\in A$ lies in $A_{-\lam'}+B_{\lam'}$. Since $A=A_{-\lam}+B_{\lam}$, we may find $a\in A_{-\lam}$ and $b\in B_{\lam}$ with
$$t=a+b.$$
Then it is clear that
$$t=\left[a+\left(1-\dfrac{\lam'}{\lam}\right)b\right]+\dfrac{\lam'}{\lam}b$$
and $\dfrac{\lam'}{\lam}b\in B_{\lam'}$, and it remains to prove that $a+\left(1-\dfrac{\lam'}{\lam}\right)b\in A_{-\lam'}$. To this end, we need to show that
$$a+\left(1-\dfrac{\lam'}{\lam}\right)b+y'\in A$$
for all $y'\in B_{\lam'}$. Indeed, $\left(1-\dfrac{\lam'}{\lam}\right)b+y'\in B_{\lam}$ since
$$\left|\left|\left(1-\dfrac{\lam'}{\lam}\right)b+y'\right|\right|\leq\left(1-\dfrac{\lam'}{\lam}\right)||b||+||y'||\leq\left(1-\dfrac{\lam'}{\lam}\right)\lam+\lam'=\lam,$$
and together with $a\in A_{-\lam}$ we deduce that $a+\left(1-\dfrac{\lam'}{\lam}\right)b+y'\in A$, which completes the proof.
\end{proof}

As an immediate consequence of Proposition \ref{LamA-closed-interval}, we have the following:

\begin{cor} \label{Alam-reg}
%$A_{-\lam_A}$ is a regular inner parallel body of $A$, whose support function is $h_{A_{-\lam_A}}=h_A-\lam_A$.
%$$h_{A_{-\lam_A}}:S^{n-1}\to\bbR,\quad h_{A_{-\lam_A}}(x)=\bw_{B\in[A]_M}h_B(x).$$
For each nonempty subset $\Lam_0\subseteq\Lam_A$, let $\lam_0=\bv\Lam_0$. Then $A_{-\lam_0}$ is a regular inner parallel body of $A$, which satisfies
$$A_{-\lam_0}=\bigcap_{\lam\in\Lam_0}A_{-\lam}\quad\text{and}\quad h_{A_{-\lam_0}}=\bw_{\lam\in\Lam_0}h_{A_{-\lam}}.$$
\end{cor}

%\begin{proof}
%First, $k(A)=A_{-\lam_A}$. If $t\in k(A)$, then $t\in A_{-\lam}$ whenever $A_{-\lam}+B_{\lam}=A$. From Lemma \ref{cb-pb-supp} we see that
%$$\langle t,x\rangle\leq h_{A_{-\lam}}=h_A-\lam$$
%\end{proof}

\begin{proof}
Firstly, with Lemma \ref{cb-pb-supp} we obtain that
$$h_{A_{-\lam_0}}=h_A-\lam_0=h_A-\bv\Lam_0=\bw_{\lam\in\Lam_0}(h_A-\lam)=\bw_{\lam\in\Lam_0}h_{A_{-\lam}}.$$

Secondly, if $t\in A_{-\lam}$ for all $\lam\in\Lam_0$, then Proposition \ref{cb-elements} implies that
$$\langle t,x\rangle\leq h_{A_{-\lam}}(x)=h_A(x)-\lam$$
for all $x\in S^{n-1}$ and $\lam\in\Lam_0$, and thus
$$\langle t,x\rangle\leq\bw_{\lam\in\Lam_0}(h_A(x)-\lam)=h_{A_{-\lam_0}}(x)$$
for all $x\in S^{n-1}$, which means that $t\in A_{-\lam_0}$. Hence $\bigcap\limits_{\lam\in\Lam_0}A_{-\lam}\subseteq A_{-\lam_0}$, which in fact becomes an identity since the reverse inclusion is trivial.
\end{proof}

In particular,
\begin{equation} \label{smallest-reg-pb}
A_{-\lam_A}=\bigcap_{\lam\in\Lam_A}A_{-\lam}
\end{equation}
is the smallest regular inner parallel body of $A$. We say that a convex body $A\in\CC^n$ is \emph{irreducible} if
$$A=A_{-\lam_A};$$
that is, if $A$ does not have any non-trivial regular inner parallel body.

%\begin{proof}
%For the ``only if'' part, suppose that $A\in\CC^n$ is not skew. Then $A=A_{-\lam}+B_{\lam}$ for some $\lam>0$, and thus $A_{-\lam_A}\subsetneq A$. Conversely, if $A_{-\lam_A}\subsetneq A$, then $A=A_{-\lam_A}+B_{\lam_A}$ and $\lam_A>0$, which means that $A$ is not skew.
%\end{proof}

Let $\Ci$ denote the set of irreducible convex bodies in $\bbR^n$, and let $\CB^n$ denote the set of closed balls in $\bbR^n$ centered at the origin. For each convex body $A\in\CC^n$, the decomposition
$$A=A_{-\lam_A}+B_{\lam_A}$$
is unique in the sense of the following:

\begin{thm} \label{cb-decomp}
For each convex body $A\in\CC^n$, there exist a unique $B\in\Ci$ and a unique $B_{\lam}\in\CB^n$ such that $A=B+B_{\lam}$. Moreover, the correspondence
$$A\mapsto(A_{-\lam_A},B_{\lam_A})$$
establishes a bijection
$$\CC^n\stackrel{\thicksim}{\longleftrightarrow}\Ci\times\CB^n,$$
whose inverse is given by $(B,B_{\lam})\mapsto B+B_{\lam}$.
\end{thm}

\subsection{Skew fuzzy vectors and Mare{\v s} cores} \label{Vec-Skew}

Motivated by the notion of \emph{skew fuzzy number} in the sense of Chai-Zhang \cite{Chai2016}, we introduce skew fuzzy vectors:

\begin{defn} \label{skew-vec-def}
A fuzzy vector $u\in\CF^n$ is \emph{skew} if it cannot be written as the sum of a fuzzy vector and a non-trivial symmetric fuzzy vector; that is, if
$$u=v\oplus w$$
for some $v\in\CF^n$ and $w\in\Fs$, then $w=\tilde{o}$.
\end{defn}

Following the terminology from fuzzy numbers \cite{Marevs1992,Hong2003,Qiu2014}, Mare{\v s} cores of fuzzy vectors are defined as follows:

\begin{defn} \label{Mares-core-def}
A fuzzy vector $v\in\CF^n$ is a \emph{Mare{\v s} core} of a fuzzy vector $u\in\CF^n$ if $v$ is skew and
$$u=v\oplus w$$
for some symmetric fuzzy vector $w\in\Fs$.
\end{defn}

These concepts are closely related to inner parallel bodies introduced in Subsection \ref{IPB}:

\begin{lem} \label{u=v-oplus-w}
For fuzzy vectors $u,v\in\CF^n$, if $u=v\oplus w$ for some symmetric fuzzy vector $w\in\Fs$, then for each $\al\in[0,1]$,
\begin{enumerate}[label={\rm(\roman*)}]
\item \label{u=v-oplus-w:level}  $v_{\al}$ is a regular inner parallel body of $u_{\al}$;
\item \label{u=v-oplus-w:supp} $h_u(\al,-)-h_v(\al,-):S^{n-1}\to\bbR$ is a (nonnegative) constant function.
\end{enumerate}
\end{lem}

\begin{proof}
Since $u=v\oplus w$ and $w$ is symmetric, Theorem \ref{vec-sum-supp} and Corollary \ref{function-sym-sup} ensure that
\begin{equation} \label{hu-al=hv-al+hw-al}
h_u(\al,-)=h_v(\al,-)+h_w(\al),
\end{equation}
and thus \ref{u=v-oplus-w:supp} holds. For \ref{u=v-oplus-w:level}, by setting $\lam=h_w(\al)$ and rewriting \eqref{hu-al=hv-al+hw-al} as $h_{u_{\al}}=h_{v_{\al}}+\lam$ it follows soon that $v_{\al}=(u_{\al})_{-\lam}$ and $u_{\al}=v_{\al}+B_{\lam}$ by Lemma \ref{cb-pb-supp}, and hence $v_{\al}$ is a regular inner parallel body of $u_{\al}$.
\end{proof}

In order to construct a Mare{\v s} core of each fuzzy vector $u\in\CF^n$, we start with the following proposition, in which
$$\Up_u:=\{v\in\CF^n\mid u=v\oplus w,\ w\in\Fs\}$$
is clearly a non-empty set as $u\in\Up_u$:

\begin{prop} \label{cu-vec}
There is a fuzzy vector $\cu\in\CF^n$ whose level sets are given by
\renewcommand\arraystretch{1.6}
$$\cu_{\al}:=\left\{\begin{array}{ll}
\bigcap\limits_{v\in\Up_u}v_{\al} & \text{if}\ \al\in(0,1],\\
\overline{\bigcup\limits_{\be\in(0,1]}\cu_{\be}} & \text{if}\ \al=0,
\end{array}
\right.$$
\renewcommand\arraystretch{1.3}%
and whose support function $h_{\cu}:[0,1]\times S^{n-1}\to\bbR$ is given by
$$h_{\cu}(\al,x)=\left\{\begin{array}{ll}
\bw\limits_{v\in\Up_u}h_v(\al,x) & \text{if}\ \al\in(0,1],\\
\lim\limits_{\be\ra 0+}h_{\cu}(\be,x) & \text{if}\ \al=0.
\end{array}
\right.$$
%In particular, $\cu\sim_M u$.
\end{prop}

\begin{proof}
For the existence of $\cu$, we show that $\{\cu_{\al}\mid\al\in[0,1]\}$ satisfies the conditions \ref{vec-cb:convex}--\ref{vec-cb:cont} in Theorem \ref{vec-cb}, as \ref{vec-cb:0} trivially holds.

Firstly, \ref{vec-cb:monotone} holds since $v_{\al}\supseteq v_{\be}$ for all $v\in\Up_u$ whenever $0\leq\al<\be\leq 1$.

Secondly, for \ref{vec-cb:convex}, note that for each $\al\in(0,1]$, Lemma \ref{u=v-oplus-w} tells us that $\cu_{\al}$ is the intersection of a family of regular inner parallel bodies of $u_{\al}$, and thus $\cu_{\al}$ is itself a regular inner parallel body of $u_{\al}$ (see Corollary \ref{Alam-reg}); in particular, $\cu_{\al}$ is a convex body. It remains to show that $\cu_0$ is a convex body. Indeed, $\cu_0$ is convex since it is the closure of the union of a family of convex bodies linearly ordered by inclusion, and its boundedness follows from $\cu_0\subseteq u_0$.

Thirdly, in order to obtain \ref{vec-cb:cont}, let $\{\al_k\}\subseteq(0,1]$ be an increasing sequence that converges to $\al_0>0$. Then $v_{\al_0}=\bigcap\limits_{k\geq 1}v_{\al_k}$ for each $v\in\Up_u$, and thus
$$\cu_{\al_0}=\bigcap\limits_{v\in\Up_u}v_{\al_0}=\bigcap\limits_{v\in\Up_u}\bigcap\limits_{k\geq 1}v_{\al_k}=\bigcap\limits_{k\geq 1}\bigcap\limits_{v\in\Up_u}v_{\al_k}=\bigcap\limits_{k\geq 1}\cu_{\al_k}.$$

For the support function of $\cu$, let $\al\in(0,1]$. Since $\cu_{\al}$ is the intersection of a family of regular inner parallel bodies of $u_{\al}$, it follows from Corollary \ref{Alam-reg} that
$$h_{\cu}(\al,-)=h_{\cu_{\al}}=\bw\limits_{v\in\Up_u}h_{v_{\al}}=\bw\limits_{v\in\Up_u}h_v(\al,-).$$
Finally, the value of $h_{\cu}(0,-)$ follows from the right-continuity of $h_{\cu}(-,x)$ at $0$ (see Theorem \ref{vec-supp}). %In particular, from the above arguments we have also deduced that  $\cu\sim_M u$ by Proposition \ref{vec-equiv-supp}.
\end{proof}

In fact, $\cu$ also lies in $\Up_u$, and it is a Mare{\v s} core of each fuzzy vector $u\in\CF^n$:

\begin{thm} \label{cu-core}
$\cu$ is skew, and there exists a symmetric fuzzy vector $\su\in\Fs$ such that $u=\cu\oplus\su$. Hence, $\cu$ is a Mare{\v s} core of $u$.
\end{thm}

\begin{proof}
{\bf Step 1.} There exists a symmetric fuzzy vector $\su\in\Fs$ such that $u=\cu\oplus\su$.

Let $h:=h_u-h_{\cu}:[0,1]\times S^{n-1}\to\bbR$. Note that for each $x\in S^{n-1}$, $h(-,x)$ is left-continuous on $(0,1]$ and right-continuous at $0$ since so are $h_u(-,x)$ and $h_{\cu}(-,x)$ by Theorem \ref{vec-supp}. Moreover, as there exists a symmetric fuzzy vector $w_v\in\Fs$ such that $u=v\oplus w_v$ for all $v\in\Up_u$, it follows from Theorem \ref{vec-sum-supp} and Proposition \ref{cu-vec} that
$$h(\al,x)=h_u(\al,x)-h_{\cu}(\al,x)=h_u(\al,x)-\bw\limits_{v\in\Up_u}h_v(\al,x)=\bv\limits_{v\in\Up_u}(h_u(\al,x)-h_v(\al,x))=\bv\limits_{v\in\Up_u}h_{w_v}(\al)$$
for all $\al\in(0,1]$, $x\in S^{n-1}$. Hence, $h$ is nonnegative, independent of $x\in S^{n-1}$ and non-increasing on $\al\in[0,1]$ because so is each $h_{w_v}$ $(v\in\Up_u)$; that is, $h$ satisfies all the conditions of Corollary \ref{function-sym-sup}. Therefore, $h$ is the support function of a symmetric fuzzy vector $\su\in\Fs$, which clearly satisfies $u=\cu\oplus\su$ by Theorem \ref{vec-sum-supp}.

{\bf Step 2.} $\cu$ is skew.

Suppose that $\cu=v\oplus w$ and $w$ is symmetric. Then $h_v\leq h_{\cu}$ by Lemma \ref{u=v-oplus-w}.

Conversely, since $w$ and $\su$ are both symmetric, so is $w\oplus\su$ by Corollary \ref{sym-vec-sum}. Thus, together with
$$u=\cu\oplus\su=v\oplus w\oplus\su=v\oplus(w\oplus\su)$$
we obtain that $v\in\Up_u$, which implies that $h_{\cu}\leq h_v$ by Proposition \ref{cu-vec}.

Therefore, $h_{\cu}=h_v$, and it forces $w=\tilde{o}$, which shows that $\cu$ is skew.
\end{proof}

An obvious application of Theorem \ref{cu-core} is to determine whether a fuzzy vector is skew:

\begin{cor} \label{skew-u=cu}
A fuzzy vector $u\in\CF^n$ is skew if, and only if, $u=\cu$.
\end{cor}

\begin{proof}
The ``if'' part is already obtained in Theorem \ref{cu-core}. For the ``only if'' part, just note that $\Up_u=\{u\}$ if $u$ is skew, and thus $u=\cu$ necessarily follows.
\end{proof}

Let $\Fk$ denote the set of skew fuzzy vectors of dimension $n$. Theorem \ref{cu-core} actually induces a surjective map as follows:

\begin{cor} \label{vec-decomp}
The assignment $(v,w)\mapsto v\oplus w$ establishes a surjective map
$$\Fk\times\Fs\to\CF^n.$$
\end{cor}

Unfortunately, as the following Example \ref{vec-deomp-not-unique} reveals, unlike Theorem \ref{fuzzy-number-main} for the case of $n=1$ or Theorem \ref{cb-decomp} for convex bodies, in general the map given in Corollary \ref{vec-decomp} may \emph{not} be injective. In other words, a fuzzy vector may have \emph{many} Mare{\v s} cores, so that there may be \emph{many} ways to decompose a fuzzy vector as the sum of a skew fuzzy vector and a symmetric fuzzy vector!

As a preparation, let us present a sufficient condition for a fuzzy vector to be skew that is easy to verify:

\begin{lem} \label{level0-irr-skew}
Let $u\in\CF^n$ be a fuzzy vector. If the $0$-level set $u_0$ of $u$ is an irreducible convex body, then $u$ is skew.
\end{lem}

\begin{proof}
Suppose that $u=v\oplus w$ and $w$ is symmetric. Then $u_{\al}=v_{\al}+w_{\al}$ for all $\al\in[0,1]$. In particular, $u_0=v_0+w_0$. Since $u_0$ is irreducible, $w_0$ must be trivial, i.e., $w_0=\{o\}$, where $o$ is the origin of $\bbR^n$. The condition \ref{vec-cb:monotone} of Theorem \ref{vec-cb} then forces $w_{\al}=\{o\}$ for all $\al\in[0,1]$, which means that $w=\tilde{o}$.
\end{proof}

\begin{exmp} \label{vec-deomp-not-unique}
Suppose that $n\geq 2$. For each $\al\in[0,1]$, let
$$A_{\al}=\prod\limits_{i=1}^n[\al-1,1-\al]$$
be the hypercube in $\bbR^n$ centered at the origin whose edge length is $2-2\al$.

For every $\lam\in[0,1]$, we may construct a fuzzy vector $v_{\lam}\in\CF^n$ whose level sets are given by
$$(v_{\lam})_{\al}=A_{\al}+B_{\lam\al},$$
and a non-trivial symmetric fuzzy vector $w_{\lam}\in\Fs$ whose level sets are given by
$$(w_{\lam})_{\al}=B_{2-\lam\al}.$$
Then there exists a fuzzy vector $u\in\CF^n$ with
$$u_{\al}=A_{\al}+B_2=A_{\al}+B_{\lam\al}+B_{2-\lam\al}=(v_{\lam})_{\al}+(w_{\lam})_{\al}$$
for all $\lam\in[0,1]$, where the second equation follows from Corollary \ref{sym-vec-sum}. Hence
$$u=v_{\lam}\oplus w_{\lam}$$
for all $\lam\in[0,1]$. As $(v_{\lam})_0=A_0=\dprod\limits_{i=1}^n[-1,1]$ is irreducible, from Lemma \ref{level0-irr-skew} we know that every $v_{\lam}$ is skew, and therefore every $v_{\lam}$ $(0\leq\lam\leq 1)$ is a Mare{\v s} core of $u$.
\end{exmp}

With Theorem \ref{fuzzy-number-main} and Example \ref{vec-deomp-not-unique} we can now conclude:

\begin{thm} \label{main}
Every fuzzy vector $u\in\CF^n$ has a unique Mare{\v s} core if, and only if, the dimension $n=1$.
\end{thm}

\subsection{Mare{\v s} equivalent fuzzy vectors} \label{Mares-Equiv}

The following definition is also originated from fuzzy numbers (see \cite{Marevs1992,Qiu2014,Chai2016}):

\begin{defn} \label{Mares-equiv}
Fuzzy vectors $u,v\in\CF^n$ are \emph{Mare{\v s} equivalent}, denoted by $u\sim_M v$, if there exist symmetric fuzzy vectors $w,w'\in\Fs$ such that
$$u\oplus w=v\oplus w'.$$
\end{defn}

The relation $\sim_M$ is clearly an equivalence relation on $\CF^n$, and we denote by $[u]_M$ the equivalence class of each $u\in\CF^n$.

\begin{prop} \label{vec-equiv-supp}
For fuzzy vectors $u,v\in\CF^n$, the following statements are equivalent:
\begin{enumerate}[label={\rm(\roman*)}]
\item \label{vec-equiv-supp:equiv} $u\sim_M v$.
\item \label{vec-equiv-supp:level} For each $\al\in[0,1]$, either $u_{\al}$ is a regular inner parallel body of $v_{\al}$, or $v_{\al}$ is a regular inner parallel body of $u_{\al}$.
\item \label{vec-equiv-supp:supp} For each $\al\in[0,1]$, the function $h_u(\al,-)-h_v(\al,-)$ is constant on $S^{n-1}$.
\end{enumerate}
\end{prop}

\begin{proof}
\ref{vec-equiv-supp:level}$\iff$\ref{vec-equiv-supp:supp} is an immediate consequence of Theorem \ref{vec-sum-supp} and Lemma \ref{cb-pb-supp}, and \ref{vec-equiv-supp:equiv}$\implies$\ref{vec-equiv-supp:supp} follows soon from Theorems \ref{vec-sum-supp} and \ref{sym-vec-supp}. For \ref{vec-equiv-supp:supp}$\implies$\ref{vec-equiv-supp:equiv}, let us fix $x_0\in S^{n-1}$, and let $\eta$ be a common upper bound of $h_u$ and $h_v$ on $[0,1]\times S^{n-1}$. Then the functions
$$\begin{array}{ccc}
[0,1]&\to&\bbR\\
\al&\mapsto&\eta+h_u(\al,x_0)
\end{array}$$
and
$$\begin{array}{ccc}
[0,1]&\to&\bbR\\
\al&\mapsto&\eta+h_v(\al,x_0)
\end{array}$$
clearly satisfy the conditions of Corollary \ref{function-sym-sup}, and thus they are support functions of symmetric fuzzy vectors $w,w'\in\Fs$, respectively.

Since $h_u(\al,-)-h_v(\al,-)$ is constant on $S^{n-1}$ for each $\al\in[0,1]$, it follows that
$$h_u(\al,x)-h_v(\al,x)=h_u(\al,x_0)-h_v(\al,x_0)=h_w(\al)-h_{w'}(\al)$$
for all $\al\in[0,1]$, $x\in S^{n-1}$; that is,
$$h_u+h_{w'}=h_v+h_w,$$
and therefore $u\oplus w'=v\oplus w$ by Theorem \ref{vec-sum-supp}, showing that $u\sim_M v$.
\end{proof}

\begin{rem}
For fuzzy vectors $u,v\in\CF^n$, $u\sim_M v$ does not necessarily imply that $u=v\oplus w$ or $v=u\oplus w$ for some $w\in\Fs$. For example, if
$$u(t)=\begin{cases}
1 & \text{if}\ t\in B_1,\\
0 & \text{if}\ t\not\in B_1
\end{cases}\quad\text{and}\quad v(t)=\begin{cases}
1-\dfrac{||t||}{2} & \text{if}\ t\in B_2,\\
0 & \text{if}\ t\not\in B_2,
\end{cases}$$
then $u,v\in\Fs$ and it trivially holds that $u\oplus v=v\oplus u$; hence $u\sim_M v$. However, there is no $w\in\Fs$ such that $u=v\oplus w$ or $v=u\oplus w$. Indeed, the level sets of $u,v$ are given by
$$u_{\al}\equiv B_1\quad\text{and}\quad v_{\al}=B_{2-2\al}$$
for all $\al\in[0,1]$; that is, $u_{\al}$ is a regular inner parallel body of $v_{\al}$ when $0\leq\al\leq\dfrac{1}{2}$, and $v_{\al}$ is a regular inner parallel body of $u_{\al}$ when $\dfrac{1}{2}\leq\al\leq 1$.
\end{rem}

It is obvious that $\cu\sim_M u$ for all $u\in\CF^n$. In fact, every Mare{\v s} core of $u$ is Mare{\v s} equivalent to $u$. In what follows we construct another skew fuzzy vector $\ku$ that is Mare{\v s} equivalent to $u$ but may not be a Mare{\v s} core of $u$:

\begin{prop} \label{ku-vec}
There is a skew fuzzy vector $\ku\in\CF^n$ whose level sets are given by
\renewcommand\arraystretch{1.6}
$$\ku_{\al}:=\left\{\begin{array}{ll}
\bigcap\limits_{v\in[u]_M}v_{\al} & \text{if}\ \al\in(0,1],\\
\overline{\bigcup\limits_{\be\in(0,1]}\ku_{\be}} & \text{if}\ \al=0,
\end{array}
\right.$$
\renewcommand\arraystretch{1.3}%
and whose support function $h_{\ku}:[0,1]\times S^{n-1}\to\bbR$ is given by
$$h_{\ku}(\al,x)=\left\{\begin{array}{ll}
\bw\limits_{v\in[u]_M}h_v(\al,x) & \text{if}\ \al\in(0,1],\\
\lim\limits_{\be\ra 0+}h_{\ku}(\be,x) & \text{if}\ \al=0.
\end{array}
\right.$$
In particular, $\ku\sim_M u$.
\end{prop}

\begin{proof}
The verification of $\ku$ being a fuzzy vector is similar to Proposition \ref{cu-vec} under the help of Proposition \ref{vec-equiv-supp}, and thus we leave it to the readers. In particular, $\ku\sim_M u$ is an immediate consequence of Proposition \ref{vec-equiv-supp} and the fact that each $\ku_{\al}$ is a regular inner parallel body of $u_{\al}$.

To see that $\ku$ is skew, suppose that $\ku=v\oplus w$ and $w$ is symmetric. Then $h_v\leq h_{\ku}$ by Lemma \ref{u=v-oplus-w}. Conversely, since $\ku\sim_M v$, it holds that $u\sim_M v$, and thus $h_{\ku}\leq h_v$. Hence $h_{\ku}=h_v$, which forces $w=\tilde{o}$ and completes the proof.
\end{proof}

\begin{rem}
For each $\al\in(0,1]$, the $\al$-level set
$$u_{\al}=\{t\in\bbR^n\mid u(t)\geq\al\}$$
of a fuzzy vector $u\in\CF^n$ has a smallest regular inner parallel body given by Equation \eqref{smallest-reg-pb} below Corollary \ref{Alam-reg}. It is then tempting to ask whether $\ku$ could be determined by the $\al$-level sets
$$\ku_{\al}=(u_{\al})_{-\lam_{u_{\al}}}.$$
Unfortunately, this is not true since, in general, $\{(u_{\al})_{-\lam_{u_{\al}}}\mid\al\in[0,1]\}$ does not satisfy the conditions of Theorem \ref{vec-cb} even when $n=1$. For example, the level sets of the fuzzy number
$$u:\bbR\to[0,1],\quad u(t)=\begin{cases}
1-\dfrac{t}{2} & \text{if}\ t\in[0,2],\\
0 & \text{else}
\end{cases}$$
are given by
$$u_{\al}=[0,2-2\al]$$
for all $\al\in[0,1]$, but
$$\{(u_{\al})_{-\lam_{u_{\al}}}\mid\al\in[0,1]\}=\{\{1-\al\}\mid\al\in[0,1]\}$$
consists of non-identical singleton sets which obviously violate the condition \ref{vec-cb:monotone}.
\end{rem}

From the definition it is clear that $\ku$ is the \emph{smallest} fuzzy vector in the Mare{\v s} equivalence class of $u\in\CF^n$; that is,
$$h_{\ku}\leq h_v$$
for all $v\in[u]_M$. In fact, if the dimension $n=1$, $\ku$ is precisely the (unique) Mare{\v s} core of a fuzzy number $u$ constructed in \cite[Proposition 4.2]{Chai2016} (cf. \cite[Remark 4.4]{Chai2016}); that is:

\begin{cor} \label{fuzzy-number-cu=ku}
For each fuzzy number $u$, it holds that
$$\ku=\cu.$$
Moreover, $\ku$ is the unique Mare{\v s} core of $u$ and the unique skew fuzzy number in the Mare{\v s} equivalence class $[u]_M$.
\end{cor}

However, if the dimension $n\geq 2$, the following continuation of Example \ref{vec-deomp-not-unique} shows that $\ku$ may not be a Mare{\v s} core of $u$, and skew fuzzy vectors may be Mare{\v s} equivalent to each other:

\begin{exmp} \label{ku-not-Mares-core}
For the fuzzy vectors $u$, $v_{\lam}$, $w_{\lam}$ $(0\leq\lam\leq 1)$ considered in Example \ref{vec-deomp-not-unique}, since
$$v_{\lam}\in[u]_M$$
for all $\lam\in[0,1]$, all skew fuzzy vectors $v_{\lam}$ are Mare{\v s} equivalent to each other. Moreover, although
$$v_0=\ku=\sk(v_{\lam}),$$
$v_0$ is not a Mare{\v s} core of $v_{\lam}$ whenever $\lam\in(0,1]$; in fact, since $v_{\lam}$ is skew, its Mare{\v s} core must be itself.
\end{exmp}

%We end this paper with the following theorem:
%
%\begin{thm} \label{Mares-class-core}
%Let $u\in\CF^n$ be a fuzzy vector. Then each $v\in[u]_M$ has the same Mare{\v s} core if, and only if,
%$$\cu=\ku.$$
%\end{thm}
%
%\begin{proof}
%Since $\cu,\ku\in[u]_M$, the ``only if'' part is obvious. For the ``if'' part,
%\end{proof}

\section{Concluding remarks}

%The Mare{\v s} core is a fundamental construction in the theory of fuzzy numbers. By applying the toolkit from convex geometry, we investigate Mare{\v s} cores of fuzzy vectors over the $n$-dimensional Euclidean space $\bbR^n$ through convex bodies and support functions, and it is shown that the following statements are equivalent:
%\begin{enumerate}[label=(\roman*)]
%\item The dimension $n=1$.
%\item Each fuzzy vector has a unique Mare{\v s} core.
%\item Each fuzzy vector can be decomposed in a unique way as the sum of a skew fuzzy vector and a symmetric fuzzy vector.
%\item There is only one skew fuzzy vector in each Mare{\v s} equivalence class of fuzzy vectors.
%\end{enumerate}

The well-known Theorem \ref{vec-cb} builds up a bridge from convex geometry to fuzzy vectors. From the viewpoint of a convex geometer, the theory of fuzzy vectors is a theory about \emph{sequences of convex bodies}. Our main result, the representation of fuzzy vectors via support functions (Theorem \ref{vec-supp}), can never be achieved without the aid of convex geometry. We believe that the power of convex geometry in the study of fuzzy vectors is still to be unveiled, which is worth further exploration. 

We end this paper with two questions about Mare{\v s} cores of fuzzy vectors that remain unsolved:
\begin{enumerate}[label=(\arabic*)]
\item Is it possible to find all the Mare{\v s} cores of a given fuzzy vector?
\item Is it possible to find a necessary and sufficient condition to characterize those fuzzy vectors with exactly one Mare{\v s} core?
\end{enumerate}

\section*{Acknowledgment}

The authors acknowledge the support of National Natural Science Foundation of China (No. 11501393 and No. 11701396), and the first named author also acknowledges the support of Sichuan Science and Technology Program (No. 2019YJ0509) and a joint research project of Laurent Mathematics Research Center of Sichuan Normal University and V.\thinspace C. \& V.\thinspace R. Key Lab of Sichuan Province.

The authors are grateful for helpful comments received from Professor Hongliang Lai and Professor Baocheng Zhu.

\bibliographystyle{abbrv}
\bibliography{lili}

\end{document}